\pdfoutput=1

\documentclass[reqno,a4paper,12pt]{amsart}
\usepackage[utf8]{inputenc}
\usepackage[T1]{fontenc}
\usepackage[english]{babel}
\usepackage[babel]{microtype}
\usepackage[a4paper,margin=1.1in]{geometry}
\usepackage[dvipsnames]{xcolor}
\usepackage{amsmath,amssymb,amsthm}
\usepackage{amsrefs}
\usepackage{mathrsfs}
\usepackage{mathtools}
\usepackage{commath}
\usepackage{thmtools}
\usepackage{bbm}
\usepackage{dsfont}
\usepackage{lmodern}
\usepackage{fixcmex}
\usepackage{enumitem}
\usepackage{centernot}
\usepackage{etoolbox}
\usepackage{tikz}
\usepackage{pgfplots}
\usepackage[pdfborderstyle={/S/U/W 0},unicode]{hyperref}
\usepackage{cleveref}

\usepackage{tikz}
\usepackage{pgfplots}
\pgfplotsset{compat=1.7}

\setlist[enumerate,1]{label=(\roman*)}
\hypersetup{
    colorlinks=true,
    linkcolor=blue!85!black,
    citecolor=blue!85!black,
    urlcolor=blue!85!black,
}
\pgfplotsset{compat=1.7}
\numberwithin{equation}{section}
\ifdefined\thmcolor
\declaretheoremstyle[
  shaded={bgcolor=\thmcolor,
  }
]{plain}
\else
\fi
\ifdefined\defcolor
\declaretheoremstyle[
  headfont=\normalfont\bfseries,
  bodyfont=\normalfont,
  shaded={bgcolor=\defcolor}
]{noital}
\else
\declaretheoremstyle[
  headfont=\normalfont\bfseries,
  bodyfont=\normalfont,
]{noital}
\fi
\declaretheorem[style=plain,numberwithin=section,name=Theorem]{theorem}
\declaretheorem[style=plain,sibling=theorem,name=Proposition]{proposition}
\declaretheorem[style=plain,sibling=theorem,name=Lemma]{lemma}
\declaretheorem[style=plain,sibling=theorem,name=Corollary]{corollary}
\declaretheorem[style=noital,sibling=theorem,name=Remark]{remark}
\declaretheorem[style=noital,sibling=theorem,name=Definition]{definition}

\newcommand{\indef}[1]{\emph{#1}}
\newcommand{\defined}{\mathrel{\coloneqq}}
\newcommand{\defines}{\mathrel{\eqqcolon}}
\DeclarePairedDelimiter{\p}{\lparen}{\rparen}
\renewcommand{\le}{\leqslant}
\renewcommand{\leq}{\leqslant}

\renewcommand{\geq}{\geqslant}
\let\oldexists\exists
\let\exists\relax
\DeclareMathOperator{\exists}{\:\!\oldexists}
\let\oldforall\forall
\let\forall\relax
\DeclareMathOperator{\forall}{\:\!\oldforall}
\newcommand{\st}{\mathbin{\colon}}
\undef{\set}
\DeclarePairedDelimiter{\set}{\lbrace}{\rbrace}
\undef{\emptyset}
\newcommand{\emptyset}{\varnothing}
\DeclarePairedDelimiter{\card}{\lvert}{\rvert}
\newcommand{\setcomp}[1]{{#1}^{\mathsf{c}}}
\newcommand{\comp}{\mathsf{c}}
\DeclareMathOperator{\ind}{\mathbf{1}}
\newcommand{\from}{\colon}
\DeclarePairedDelimiter{\floor}{\lfloor}{\rfloor}
\DeclarePairedDelimiter{\ceil}{\lceil}{\rceil}
\renewcommand{\d}{\mathop{}\!\mathrm{d}}
\newcommand{\dt}{\d t}
\newcommand{\grad}{\nabla}
\newcommand{\hess}{\nabla^2}
\DeclareMathOperator{\loglog}{log\,log}
\undef{\abs}
\DeclarePairedDelimiterX{\abs}[1]
  {\lvert}{\rvert}{\ifblank{#1}{\,\cdot\,}{#1}}
\undef{\norm}
\DeclarePairedDelimiterX{\norm}[1]
  {\lVert}{\rVert}{\ifblank{#1}{\,\cdot\,}{#1}}
\DeclarePairedDelimiterX{\inner}[2]
  {\langle}{\rangle}{\ifblank{#1}{\,\cdot\,}{#1},\ifblank{#2}{\,\cdot\,}{#2}}
\newcommand*{\conj}[1]{\overline{#1}}
\renewcommand{\Re}{\operatorname{Re}}
\renewcommand{\Im}{\operatorname{Im}}

\newcommand{\tran}{\top}
\DeclareMathOperator{\trace}{Tr}

\DeclareMathOperator{\Bin}{Bin}
\DeclareMathDelimiter{\given}
  {\mathbin}{symbols}{"6A}{largesymbols}{"0C}
\DeclareMathOperator{\Prob}{\mathbb{P}}
\DeclarePairedDelimiterXPP{\prob}[1]
  {\Prob}{\lparen}{\rparen}{}
  {\renewcommand{\given}{\nonscript\;\delimsize\vert\nonscript\;\mathopen{}}#1}
\DeclareMathOperator{\Expec}{\mathbb{E}}
\DeclarePairedDelimiterXPP{\expec}[1]
  {\Expec}{\lparen}{\rparen}{}
  {\renewcommand{\given}{\nonscript\;\delimsize\vert\nonscript\;\mathopen{}}#1}
\newcommand{\eps}{\varepsilon}
\let\SS\relax
\newcommand{\CC}{\mathbb{C}}

\newcommand{\RR}{\mathbb{R}}
\newcommand{\SS}{\mathbb{S}}
\newcommand{\cE}{\mathcal{E}}
\newcommand{\cC}{\mathcal{C}}
\newcommand{\bfM}{\mathbf{M}}

\newcommand{\degmin}{d_{\min}}
\newcommand{\degmax}{d_{\max}}
\renewcommand{\deg}{d}

\newcommand{\anorm}{\alpha}
\newcommand{\degm}{d^*}
\newcommand{\lnorm}{c}

\newcommand{\Sin}{H}

\newcommand{\ER}{Erd\H{o}s--Rényi\ }

\begin{document}

\title{Expander graphs are globally synchronizing}

\author[Abdalla]{Pedro Abdalla}
\author[Bandeira]{Afonso S. Bandeira}
\author[Kassabov]{Martin Kassabov}
\author[Souza]{Victor Souza}
\author[Strogatz]{Steven H. Strogatz}
\author[Townsend]{Alex Townsend}

\address{Department of Mathematics, UC Irvine}
\email{pabdalla@ad.uci.edu}
\address{Department of Mathematics, ETH Zürich}
\email{bandeira@math.ethz.ch}
\address{Department of Mathematics, Cornell University, Ithaca, NY 14853}
\email{martin.kassabov@cornell.edu}
\email{strogatz@cornell.edu}
\email{townsend@cornell.edu}

\address{Department of Mathematics, Cornell University, Ithaca NY 14853, USA; and Sidney-Sussex College, Cambridge, CB2 3HU, United Kingdom}
\email{vsouza@cornell.edu, vss28@cam.ac.uk}

\begin{abstract}
The Kuramoto model is fundamental to the study of synchronization. 
It consists of a collection of oscillators with interactions given by a network, which we identify respectively with vertices and edges of a graph.
In this paper, we show that a graph with sufficient expansion must be globally synchronizing, meaning that a homogeneous Kuramoto model of identical oscillators on such a graph will converge to the fully synchronized state with all the oscillators having the same phase, for every initial state up to a set of measure zero.
In particular, we show that for any $\varepsilon > 0$ and $p \geq (1 + \varepsilon) (\log n) / n$, the homogeneous Kuramoto model on the Erd\H{o}s--Rényi random graph $G(n, p)$ is globally synchronizing with probability tending to one as $n$ goes to infinity.
This improves on a previous result of Kassabov, Strogatz, and Townsend and solves a conjecture of Ling, Xu, and Bandeira.
We also show that the Kuramoto model is globally synchronizing on any $d$-regular Ramanujan graph, and on typical $d$-regular graphs, for $d \geq 600$.
\end{abstract}

\maketitle


\section{Introduction}
\label{sec:intro}

In 1665, Christiaan Huygens---the inventor of the pendulum clock---observed that two pendulum clocks spontaneously synchronize when hung from the same board.
This is the first recorded observation of synchronization of coupled oscillators.
In the centuries since then, synchronization has become a central subject of study in dynamical systems, motivated both by its mathematical interest and its many applications in classical mechanics, statistical physics, neuroscience, engineering, medicine, and biology~\cites{Acebron2005-lt, Arenas2008-sr, Bick2020-zp, Dorfler2014-kx, Pikovsky2015-jl, Pikovsky2001-lm, Rodrigues2016-ez, Strogatz2000-le, Strogatz2003-fh}.

We are interested in understanding spontaneous synchronization of multiple coupled oscillators whose pairwise connections are given by a graph $G = (V,E)$.
Following the celebrated work of physicist Yoshiki Kuramoto in 1975~\cite{Kuramoto1975-tt}, we consider the following system of ordinary differential equations:
\begin{equation}
\label{eq:kuramoto}
    \frac{\d \theta_x}{\dt} = \omega_x - \sum_{y \in V} A_{x,y} \sin(\theta_x - \theta_y),
    \qquad \text{for all $x \in V$,}
\end{equation}
where each vertex $x\in V$ has associated with it an oscillator $\theta_x \from \RR \to \SS^1$.
Here $\theta_x(t)$ is the phase of oscillator $\theta_x$ at time $t$,  $\omega_x \in \RR$ is the oscillator's intrinsic frequency, and $A$ is the symmetric adjacency matrix of the graph $G$.

The Kuramoto model~\eqref{eq:kuramoto} is an $n$-dimensional nonlinear dynamical system, where $n$ is the number of vertices in $G$. For simplicity, the early work on this model focused on the case where all $n$ oscillators interact with each other with equal strength, corresponding to a complete graph on $n$ vertices. To account for the  diversity inherent in real biological oscillators, the intrinsic frequencies $\omega_x$ were typically assumed to vary at random across the oscillator population according to a prescribed probability density $g(\omega)$, taken to be unimodal and symmetric about its mean. In the limit $n \to \infty$, Kuramoto~\cites{Kuramoto1975-tt, Kuramoto1984-qz} showed that the model undergoes a phase transition (from a completely disordered state to a partially synchronized state) if the oscillators are close enough to identical, i.e., if the width of $g(\omega)$ falls below a certain threshold.
For reviews of the Kuramoto model on a complete graph, see~\cites{Acebron2005-lt, Bick2020-zp, Pikovsky2015-jl, Strogatz2000-le}. 

Recent work on the Kuramoto model has turned to the exploration of other interaction graphs. The goal is to understand which networks foster synchronization.
We focus our analysis on the \emph{homogeneous} case where all the oscillators have the same intrinsic frequency $\omega_x = \omega$.
The advantage of the homogeneous case is that the oscillators are then guaranteed to settle down to phase-locked states; in other words, all the phase differences $\theta_x(t)-\theta_y(t)$ asymptotically approach constant values as $t \to \infty$.
In contrast, if the oscillators' frequencies are non-identical, more complicated long-term behavior such as chaos can occur~\cite{Maistrenko2005-os}. 

For the homogeneous case, it proves convenient to recast the dynamics in a rotating frame $\theta(t) \gets \theta(t) - \omega t$.
Then \eqref{eq:kuramoto} is given by the gradient flow $\d\theta / \dt = - \grad \cE_G(\theta)$, in which the oscillators can be viewed as attempting to minimize the energy $\cE_G \from (\SS^1)^V \to \RR$ defined by
\begin{equation}
\label{eq:energy}
    \cE_G(\theta) = \frac{1}{2} \sum_{x,y \in V} A_{x,y}\p[\big]{1 - \cos(\theta_x - \theta_y)}.
\end{equation}
The long-term dynamics can thus be described by the landscape of the energy function $\cE_G$.
A state $\theta$ is called an \indef{equilibrium state} if $\d \theta / \dt = 0$.
In view of~\eqref{eq:kuramoto}, all the oscillators in an equilibrium state move at the intrinsic frequency $\omega$ in the original reference frame, without changing their relative phase differences, whereas in the rotating frame, all the oscillators would be motionless.
As the dynamics follow a gradient system, the equilibrium states are the critical points of the energy, namely, states $\theta$ such that $\grad \cE_G(\theta) = 0$.
If, furthermore, $\theta$ is a local minimum of $\cE_G$, we call $\theta$ a \indef{stable equilibrium state}, or simply a \indef{stable state}\footnote{Since $\cE_G$ is analytic, the local minima $\theta$ are precisely the states that are (Lyapunov) stable, meaning that for every $\eps > 0$, there is $\delta > 0$ such that $\norm{\theta(0) - \theta} \leq \delta$ implies $\norm{\theta(t) - \theta} \leq \eps$ for all $t \geq 0$. See Absil and Kurdyka~\cite{Absil2006-lg}.}.

The energy $\cE_G(\theta)$ is always non-negative and it is minimized when all the oscillators have the same phase, that is, $\theta_x = \theta_y$ for all $x,y \in V$.
These states are called the \indef{fully synchronized states}.\footnote{We consider two states $\theta$ and $\theta'$ to be \indef{equivalent} if they differ by a global rotation, that is $\theta_x - \theta'_x = c$ for all $x\in V$, for some $c \in \RR$.}
If  other local minima of $\cE_G$ exist, they are called \indef{spurious local minima}.
The main goal of this paper is to understand, for which graphs $G$, the energy landscape $\cE_G$ has no spurious local minima.

\begin{definition}[Globally synchronizing graph]
Consider the Kuramoto model on a graph $G$ with an initial state $\theta$ chosen uniformly at random from $(\SS^1)^V$.
We say that $G$ is \indef{globally synchronizing} if $\theta$ converges to a fully synchronized state with probability one.
\end{definition}

If $\cE_G$ has a spurious local minimum then there will be an open set of initial configurations that would converge to said minimum, so $G$ is not globally synchronizing.
Even if $\cE_G$ has no spurious local minima, $G$ may not be globally synchronizing as other types of undesirable equilibrium points are possible.
These can be, for instance, higher-order saddle points of $\cE_G$ where the Hessian $\hess \cE_G$ is zero\footnote{A simple example is given by the  potential $U(x,y) = x^3 + y^3$. Indeed, convergence to the saddle point $(0,0)$ occurs when starting from any point $(x,y)$ in the positive quadrant.}.

Nevertheless, each trajectory for our system converges to a single critical point of $\cE_G$.
No other form of long-term behavior is possible, because the homogeneous Kuramoto model is an instance of a gradient flow of an analytic potential function on a compact analytic manifold, and the solutions of such flows are always convergent to single points, as a consequence of the Łojasiewicz gradient inequality, see Lageman~\cite{Lageman2007-ps}.
As $\cE_G$ is analytic, the set of initial states from which the Kuramoto model would converge to a strict saddle point of $\cE_G$ (that is, states $\theta$ such that the Hessian $\hess \cE_G (\theta)$ has a strictly negative eigenvalue) has measure zero, see e.g. \cite{Geshkovski2023-tn}*{Lemma A.1}.
The upshot is that to show that a graph $G$ is globally synchronizing, it is enough to guarantee that the only equilibrium states $\theta$ with positive semidefinite Hessian $\hess \cE_G (\theta)$ are the fully synchronized states.

An active line of work has focused on determining which graphs are globally synchronizing.
Random graph models are of particular importance, as they play the role of proxies for realistic complex networks.
The classical model for a random graph is the \ER model: a graph $G$ drawn from $G(n,p)$ is a graph on $n$ vertices where each edge is present, independently, with probability $p$.
Ling, Xu, and Bandeira in~\cite{Ling2019-yn} established a connection between $\cE_G(\theta)$ and non-convex optimization landscapes arising in algorithms for community detection and other statistical inference problems~\cites{Bandeira2016-qr}, which allowed them to show that an \ER graph with edge probability $p = \Omega(n^{-1/3} \log n)$ is globally synchronizing with high probability.\footnote{That is, with probability tending to one as $n$ tends to infinity.}
The result was improved to $p = \Omega((\log n)^2/n )$ by Kassabov, Strogatz, and Townsend~\cite{Kassabov2022-nf}.
Ling, Xu, and Bandeira in~\cite{Ling2019-yn}*{Open Problem 3.4} formulated a conjecture that almost any graph in $G(n,p)$ is globally synchronizing, where $p = (1 + \eps) (\log n) / n$ and $\eps > 0$. As this corresponds to the connectivity threshold of $G(n,p)$, this threshold is the best possible.
One of our main results is to prove this conjecture.

\begin{theorem}
\label{thm:ling-xu-bandeira-conjecture}
A graph in $G(n,p)$ is globally synchronizing with high probability whenever $p \geq (1 + \eps) (\log n) / n$ and $\eps > 0$.
In other words, if a graph $G$ has $n$ vertices and each possible edge is independently included in $G$ with probability $p \geq (1 + \eps) (\log n) / n$, then the probability that $G$ is globally synchronizing tends to $1$ as $n$ tends to infinity.
\end{theorem}

We prove global synchronization of graphs drawn from $G(n,p)$ by using their expansion properties and proving that graphs that are sufficiently good expanders necessarily globally synchronize.
From a bird's-eye view, the reason is that spurious stable fixed points need to correspond, in a certain sense, to assignments from the vertices to phases such that pairs of vertices corresponding to similar phases are connected significantly more often than pairs of vertices corresponding to disparate phases.
However, expander graphs are known to satisfy versions of \emph{expander mixing lemmas} that force most pairs of subsets of vertices to have a number of edges connecting them that is close to the expected number of edges if the edges were placed randomly.
This rules out stable states formed of phases that are well spread on the circle, as properties of expander graphs prevent the possibility of having pairs of similar phases be ``over-connected,'' as compared to pairs of disparate phases.
At the other extreme, nontrivial stable fixed points that are formed of phases only occupying less than half of the circle can be ruled out by a standard result known as the half-circle lemma (intuitively, if all phases are in the same half circle then, on average, the ``forces'' experienced by the particles push towards directions in that half-circle and cannot average out to zero, ruling out stationarity, see \Cref{lem:half-circle}). 

One core difficulty is that the two types of arguments are fundamentally different: the former, for well-spread phases, is based on second-order (stability) conditions while the latter, for concentrated phases, is based on first-order (stationary) conditions.
Kassabov, Strogatz, and Townsend~\cite{Kassabov2022-nf} developed a so-called \indef{amplification argument} that can handle candidate stable fixed points in between these two regimes by simultaneously leveraging both strategies; this allowed them to show global synchrony for $p = \Omega((\log n)^2/n )$.
Our approach builds on~\cite{Kassabov2022-nf} and achieves the connectivity threshold by developing an improved version of the amplification argument that is based on a new stability condition, and by also performing a more refined analysis of the expansion and spectral properties of the graph.

\subsection{Proof overview}

Our \emph{amplification argument} (\Cref{sec:amplification}) works the following way:
Given a stable state $\theta$, we interpret each phase $\theta_x$ as a point mass at $e^{i\theta}$ on the unit circle in the complex plane and consider the half-circle $\cC$ with the smallest mass.
If the arc $\cC$ is empty, then $\theta$ is a fully synchronized state, by the half-circle lemma (\Cref{lem:half-circle}).
Otherwise, we progressively enlarge the arc $\cC$ to pick up more and more of the mass.
The core of the argument is the use of \indef{amplification steps}, where we exploit the properties of the expander graph to show that the mass encompassed by the arc grows exponentially fast without growing the arc itself too quickly.
In each of these steps, we can grow the arc by $\delta$ radians while growing the mass by a multiplicative factor of $1 + f(\delta)$.
By repeatedly applying these amplification steps, we obtain a new arc with most of the mass.
If done carefully, this argument leads to a contradiction when the mass of the final arc is put together with various stability conditions of the Kuramoto model.
In essence, if most of the mass is on the final arc, and the final arc is not too far from the original half-circle, then the original half-circle cannot have been the one with the smallest mass.
The only possibility then is that the half-circle was originally empty and $\theta$ is fully synchronized.\footnote{This technique resembles the \indef{density increment argument} in extremal and additive combinatorics (see the blog of Tao~\cite{Tao2009-hd} for several examples) and while it takes considerable input from the Kuramoto model, it may be possible that it can be adapted to other models and phenomena beyond synchronization.}

Among the new technical tools we developed to prove \Cref{thm:ling-xu-bandeira-conjecture}, we highlight:
\begin{enumerate}
    \item A novel kernel stability condition (\Cref{lem:kernel-stability}) for the Kuramoto model.
    This was a crucial piece in obtaining different amplification steps (\Cref{lem:amplification-simple} and \Cref{lem:amplification-cancelation}) that provide a massive savings on the late stages of the argument.
    Without this improvement, we would not be able to obtain our result with $p = C (\log n)/n$, even with an arbitrarily large constant $C$.
    \item A refined two-sided notion of expansion, together with its associated spectral graph theory machinery.
    When $p = C(\log n)/n$, not only do the degrees of the \ER graph fail to concentrate, but the maximum and minimum degree do not deviate symmetrically from the mean.
    We must pay special attention to vertices of small degree as they can interfere with expansion on small sets.
    With one-sided control of expansion, we would not be able to obtain our global synchrony in $G(n,p)$ result when $p < 2.58 (\log n)/n$.
    \item We also employ techniques from random matrix theory in order to obtain explicit control on the eigenvalues of the adjacency matrix of $G(n,p)$.
    This allows us, for instance, to obtain an explicit bound on the probability that $G(n,p)$ is not globally synchronizing; see \Cref{thm:gnp-spectral}, as well as further details in \Cref{sec:erdos}.
\end{enumerate}

The only information from the underlying graph needed to implement our version of the amplification argument is a control of the spectrum of the adjacency and Laplacian matrices. Indeed, our main contribution is to show that expander graphs, with sufficiently large expansion, are globally synchronizing.
Expander graphs are a central object of study in graph theory and theoretical computer science (see, e.g.,~\cite{Hoory2006-uk}) as they are good sparse surrogates for complete graphs and have many pseudorandom properties. 

After introducing the relevant notation, we state all our results in \Cref{subsec:regular} and \Cref{subsec:nonregular} below.
In \Cref{subsec:related} we briefly outline some related work on globally synchronizing graphs.

\subsection{Notation}

A graph $G = (V,E)$ with adjacency matrix $A \in \set{0,1}^{n \times n}$ is said to be $d$-regular if all vertices have degree $d$.
The eigenvalues of $A$ are represented by $\lambda_1(A) \geq \dotsb \geq \lambda_n(A)$.
We denote by $\ind$ the all-ones vector, by $J = \ind \ind^\tran$ the all-ones matrix and by $I$ the identity matrix, always taken with appropriate sizes.
For symmetric matrices $A$ and $B$ of same dimensions, the notation $A\succeq B$ means that $A-B$ is a positive semidefinite matrix, i.e., the symbol $\succeq$ denotes the standard Loewner partial order.
The degree matrix $D \in \RR^{n \times n}$ is a diagonal matrix with $D_{xx} = \deg(x)$ and the graph Laplacian is given by $L = D - A$.
We define the \indef{centered adjacency matrix} by $\Delta_A \defined A - (\degm/n) J$, the \indef{centered diagonal degree matrix} $\Delta_D \defined D - \degm I$ and the \indef{centered Laplacian matrix} $\Delta_L \defined \Delta_D - \Delta_A = L - \degm I + (\degm / n)J$.
Here, $\degm$ is a parameter to be chosen, which we always select as an ``expected average degree'' of $G$.
For instance, when $G$ is $d$-regular, we take $\degm = d$, and for $G(n,p)$, we take $\degm = pn$.
We write $\norm{H}$ for the operator norm of a matrix $H$, namely, $\norm{H} = \sup_{f \neq 0} \norm{Hf}/\norm{f}$.

\subsection{Main results: regular graphs}
\label{subsec:regular}

For regular graphs, our main result is that spectral expansion implies global synchrony.

\begin{definition}
We say that a graph $G$ is an $(n, \degm, \anorm)$-expander if it has $n$ vertices and
\begin{equation*}
    \norm[\big]{\Delta_A} \defined \norm[\big]{A - (\degm / n) J} \leq \anorm \degm.
\end{equation*}
If in addition the graph $G$ is $\degm$-regular\footnote{In the literature, $d$-regular $(n,d,\lambda d)$-expanders are commonly called $(n,d,\lambda)$-graphs.}, then the condition above is equivalent to $\max_{i \neq 1} \abs[\big]{\lambda_i(A)} \leq \anorm \degm$.
\end{definition}

For regular graphs, our main result is the following.

\begin{theorem}
\label{thm:main-synchrony-regular}
All $d$-regular $(n, d,\anorm)$-expander graphs with $\anorm \leq 0.0816$ are globally synchronizing.
\end{theorem}
\begin{remark}
The proof of Theorem \ref{thm:main-synchrony-regular} relies on some numerical computations outlined in \Cref{sec:numerics}.
Without such computation we can only prove a weaker version for $\anorm \leq 0.0068$, which follows from \Cref{thm:amplification-main}.
\end{remark}

An important class of expander graphs are the Ramanujan graphs.
These are $d$-regular $(n, d, \anorm)$-expanders with $\anorm = 2(\sqrt{d-1}) / d$, which is asymptotically the smallest possible value of $\anorm$ for a given $d$, due to a result of Alon and Boppana~\cite{Nilli1991-ga}.
In other words, $d$-regular Ramanujan graphs are the best expanding graphs among the class of $d$-regular graphs.

While the construction of explicit Ramanujan graphs is not an easy endeavor, some explicit constructions \cites{Margulis1988-bx,Lubotzky1988-gu} and some partial results about their existence are known~\cite{Marcus2015-pt}.
Moreover, random $d$-regular graphs have been shown to be almost Ramanujan.
Indeed, solving a conjecture by Alon, Friedman~\cites{Friedman2008-sf} proved that for any $d \geq 3$ and $\eps > 0$, a uniformly chosen $d$-regular random graph on $n$ vertices is, with high probability, an $(n, d, \anorm)$-expander with $\anorm = (2\sqrt{d-1} + \eps)/d$. 
Since our results do not depend on the delicate condition of being a Ramanujan graph, but only on expansion, they will also hold for random $d$-regular graphs.

\begin{corollary}
\label{cor:main-ramanujan}
Any $d$-regular Ramanujan graph is globally synchronizing as long as $d \geq 600$.
Moreover, a random $d$-regular graph is globally synchronizing with high probability in the same range of degrees $d$.
\end{corollary}
\begin{proof}
Recall that Ramanujan graphs are $d$-regular $(n, d, \anorm)$-expanders with $\anorm = 2 (\sqrt{d-1})/ d$.
The first conclusion follows from \Cref{thm:main-synchrony-regular} as $2(\sqrt{d-1}) / d \leq 0.0816$ for $d \geq 600$.
The second part follows from Friedman's result mentioned above.
\end{proof}

\begin{remark}
Expansion is not a necessary condition for global synchronization.
In fact, a graph consisting of two cliques connected by an edge is globally synchronizing, even though it has poor expansion.
Indeed, this follows from a more general fact that if two graphs $H$ and $H'$ are such that $\cE_H$ and $\cE_{H'}$ have no spurious local minima and all their saddle points are strict saddle points, then the same holds for the graph obtained by gluing a vertex of $H$ to a vertex of $H'$ (see Canale and Monzon~\cite{Canale2007-oe}).
Finally, it can be shown that the only states $\theta$ for which the energy of the complete graph $K_n$ satisfy $\hess \cE_{K_n}(\theta) \succeq 0$ are the fully synchronized states (see Taylor~\cite{Taylor2012-jc}).

\end{remark}

\begin{remark}
Although more connections intuitively help synchronization, this is not always the case.
In fact, it is not hard to see that any tree is globally synchronizing (from the argument above, for instance), while a cycle of length greater than five is not globally synchronizing~\cites{Wiley2006-fz, Taylor2012-jc, Townsend2020-fa}.
\end{remark}

\subsection{Main results: non-regular graphs}
\label{subsec:nonregular}

As one of our main goals is to establish global synchrony of \ER graphs near the connectivity threshold, we need to handle graphs with a significant degree disparity.
Our notion of expansion needs to include also spectral properties of the graph Laplacian, and due to the asymmetry of the spectral deviations of the graph Laplacian of an \ER graph close to the connectivity threshold (see \Cref{fig:implicit-equation}), we need an asymmetric definition of expansion.

\begin{definition}[Spectral expander graph]
A graph $G$ is an $(n, \degm ,\anorm, \lnorm^-, \lnorm^+)$-expander if it is an $(n, \degm, \anorm)$-expander and it holds that
\begin{equation*}
    \lnorm^- \degm I \preceq \Delta_L \preceq \lnorm^+ \degm I.
\end{equation*}
\end{definition}

We emphasize that an $(n, \degm, \anorm)$-expander need not be regular.
On the other hand, a $\degm$-regular $(n, \degm, \anorm)$-expander is also an $(n, \degm, \anorm, -\anorm, \anorm)$-expander, as in this case $\Delta_L = - \Delta_A$.
Our main result is the following technical theorem.

\begin{theorem}
\label{thm:amplification-main}
If $G$ is an $(n, \degm, \anorm, \lnorm^-, \lnorm^+)$-expander graph with $\lnorm^- > -1$, $\anorm \leq 1/5$ and
\begin{equation}
\label{eq:main-condition}
    \frac{32 \anorm (1 + 4\lnorm^+ - 3\lnorm^-) }{(1 + \lnorm^-)^2} \, \log\p[\Big]{\frac{1 + \lnorm^+ + \anorm}{2\anorm}} < 1,
\end{equation}
then $G$ is globally synchronizing.
\end{theorem}

As a corollary, we solve the conjecture of Ling, Xu, and Bandeira.
Quantitative estimates on the expanding properties of $G(n,p)$ are provided in \Cref{sec:erdos}.

\begin{proof}[Proof of \Cref{thm:ling-xu-bandeira-conjecture}]
By \Cref{cor:expansion-whp}, if $p \geq (1+\eps)( \log n) / n $, we have that $G(n,p)$ is an $(n, \degm, \anorm, \lnorm^-, \lnorm^+)$-expander with high probability, where $\degm = pn$, $\anorm \leq (2+\eps) (pn)^{-1/2}$ and $\lnorm^- = \lnorm^-(\eps)$, $\lnorm^+ = \lnorm^+(\eps)$ satisfy $-1 < \lnorm^- < 0 < \lnorm^+ < e-1$.
From \Cref{thm:amplification-main}, $G(n,p)$ is globally synchronizing as long as~\eqref{eq:main-condition} is satisfied, which holds with high probability as $n \to \infty$ (since $(pn)^{-1/2} \log(pn) \to 0$).
\end{proof}

\subsection{Related work}
\label{subsec:related}

It is natural to ask how the density of a graph influences whether or not it is globally synchronizing.
A practical metric for assessing density is the minimum degree.
It is interesting, then, to determine the critical connectivity threshold $\mu_c$: the smallest number for which, for any $\mu > \mu_c$ and $n$ large enough, any graph on $n$ vertices with minimum degree $\degmin \geq \mu(n-1)$ must globally synchronize.

After a series of both upper and lower bounds provided by
Taylor~\cite{Taylor2012-jc},
Canale and Monzón~\cite{Canale2015-ez},
Ling, Xu, and Bandeira~\cite{Ling2019-yn},
Lu and Steinerberger~\cite{Lu2020-ax},
Townsend, Stillman, and Strogatz~\cite{Townsend2020-fa},
and Yoneda, Tatsukawa, and Teramae~\cite{Yoneda2021-ms},
it is now known that $0.6875 \leq \mu_c \leq 0.75$.
The lower bound is due to Canale~\cite{Canale2022-iy}, and the upper bound is due to Kassabov, Strogatz, and Townsend~\cite{Kassabov2021-rg}.
The lower bound implies that a graph merely being connected is far from being a sufficient condition for global  synchronization, but the upper bound reflects the intuitive idea that a ``well connected'' graph should synchronize.

Conversely, an intriguing question arises: how sparse can a globally synchronizing graph be?
As previously mentioned, a tree is globally synchronizing, whereas a long cyclic graph is not.
This indicates that except for incredibly dense graphs, the topological structure of the graph is more important than mere density.
Indeed, the presence of long cycles seems to be crucial to the existence of stable states other than the fully synchronized one.
Furthermore, while a long cyclic graph has spurious stable states, \Cref{cor:main-ramanujan} implies that except for the fully synchronized states, the other stable states are fragile in the following sense:
If one were to add a set $S$ of edges to the long cycle $C_n$ with the goal of destroying all the spurious local minima, one could do it by adding no more than $600$ (random) edges incident to every given vertex.
This follows from the fact that a random $d$-regular graph is Hamiltonian for large $n$ and $d \geq 3$, as shown by Robinson and Wormald~\cite{Robinson1994-fh}.
Indeed, this implies that the edges of a $600$-regular random graph on $n$ vertices can be partition as $S \sqcup H$, where $H$ is said Hamiltonian cycle and $S$ is a collection of edges with the property that every vertex is incident to $< 600$ edges in $S$.

Recently, an interesting connection was made between dynamics under the Kuramoto model and dynamics arising in transformer architectures in deep learning, corresponding essentially to gradient dynamics under a different potential.
We point the reader to the work of Geshkovski, Letrouit, Polyanskiy, and Rigollet~\cite{Geshkovski2023-tn} for a thorough description of this connection.

\subsection{Outline}

The rest of the paper is organized as follows.
In \Cref{sec:kuramoto-graph} we develop preliminary estimates for the energy function associated with the homogeneous Kuramoto model.
\Cref{sec:edgebounds} presents some spectral graph machinery.
In \Cref{sec:amplification}, we prove the main result \Cref{thm:amplification-main} and, in \Cref{sec:numerics}, we discuss how to improve the bound for $\anorm$ in \Cref{thm:main-synchrony-regular} using numerical optimization.
\Cref{sec:erdos} provides the probabilistic estimates needed for the \ER model.

\section{Preliminaries for the Kuramoto model on a graph}
\label{sec:kuramoto-graph}

In this section, we derive a few preliminary results about the energy function $\cE_G$ associated with our Kuramoto model.
To keep the notation simple we only work with graphs having unweighted edges; all results in the paper can be extended to graphs with weighted edges, provided that all the weights are non-negative.

Recall that our goal is to study how the topology of $G$ influences the \emph{energy landscape} of $\cE_G$ given by
\begin{equation*}
    \cE_G(\theta) = \frac{1}{2} \sum_{x,y \in V} A_{x,y}\p[\big]{1 - \cos(\theta_x - \theta_y)}.
\end{equation*}
A simple computation shows that $\cE_G$ has gradient
\begin{equation}
\label{eq:gradient}
    \p[\big]{ \grad \cE_G(\theta) }_x = \sum_{z \in V} A_{x,z} \sin(\theta_x - \theta_z),
\end{equation}
and the Hessian is given by
\begin{equation}
\label{eq:Hessian}
    \p[\big]{ \hess \cE_G(\theta) }_{x,y} =
    \begin{cases}
        -A_{x,y}\cos(\theta_x - \theta_y) & \text{ if $x \neq y$,} \\
        \displaystyle\sum_{z \in V \setminus \set{x}} A_{x,z}\cos(\theta_x - \theta_z) & \text{ if $x = y$.}
    \end{cases}
\end{equation}
Therefore, every critical point $\theta$ of $\cE_G$ satisfies the following \indef{equilibrium condition}:
\begin{equation}
\label{eq:equilibrium}
	\sum_{x \in V} A_{x,y} \sin(\theta_x - \theta_y) = 0, \qquad \forall y\in V.
\end{equation}
An equivalent way to phrase the equilibrium condition is that, for each $y \in V$, there is $r_y \in \RR$ such that
\begin{equation}
\label{eq:equilibrium-complex}
	\sum_{x \in V} A_{x,y} e^{i \theta_x} =  r_y e^{i \theta_y}.
\end{equation}

In addition, if $\theta$ is a local minimum of $\cE_G$ (also called a stable state), then a \indef{stability condition} must hold.
That is, the Hessian $\hess \cE_G(\theta)$ is positive semidefinite.
Recall that a Hermitian matrix $H$ is positive semidefinite if $\inner{f}{Hf} \geq 0$ for every $f \from V \to \CC$, where $\inner{}{}$ is the usual Hermitian inner product $\inner{f}{g} = \sum_{x \in V} f(x) \conj{g(x)}$.
In general, the stability condition can be written as
\begin{align*}
    \inner{f}{\hess \cE_G(\theta)f}
    &= \sum_{x,y \in V} A_{x,y} \cos(\theta_x - \theta_y) f(x) \conj{\p[\big]{f(x) - f(y)}} \\
    &= \frac{1}{2}\sum_{x,y \in V} A_{x,y} \cos(\theta_x - \theta_y) \abs[\big]{f(x) - f(y)}^2 \geq 0.
\end{align*}
Thus, taking $f(x) = \exp(i\theta_x)$, we obtain
\begin{align}
\label{eq:cosine-condition}
    \sum_{x,y \in V} A_{x,y} \cos(\theta_x - \theta_y) \p[\big]{1 - \cos (\theta_x - \theta_y) } \geq 0.
\end{align}
A similar condition was used by Ling, Xu, and Bandeira~\cite{Ling2019-yn}*{p.~1893}.
Another useful consequence is derived by taking an indicator function $f = \ind_y$, namely, that $r_y \geq 0$ (see~\eqref{eq:equilibrium-complex}), i.e.,
\begin{equation}
\label{eq:ry-positivity}
	r_y = \sum_{x \in V} A_{x,y} \cos(\theta_x - \theta_y)  \geq 0.
\end{equation}

\subsection{Stable states and their Daido order parameters}
\label{sec:order-parameters}

The so-called \indef{Daido order parameters} $\rho_k(\theta)$ \cite{Daido1992-td}, defined by
\begin{equation}
\label{eq:order-parameter}
	\rho_k(\theta) \defined \frac{1}{\card{V}} \sum_{x \in V} e^{i k \theta_x}, \qquad k\geq 1,
\end{equation}
are fundamental statistics in the analysis of the Kuramoto model.
It is clear that for each $k$ and each state $\theta$, the parameter $\rho_k$ lies in the complex unit disc, namely $\abs{\rho_k} \leq 1$.
If $\abs{\rho_1(\theta)} = 1$, then the state $\theta$ must be the all-in-phase state where the oscillators are in perfect unison.
In some sense, the closer $\abs{\rho_1(\theta)}$ is to $1$, the closer the state is to the all-in-phase state.
Moreover, the phase of $\rho_1(\theta)$ indicates the average phase of the oscillators.
Since our Kuramoto model is invariant to global rotation of the phases, we can shift the phases so that $\rho_1(\theta)$ is real and non-negative.
For the rest of this paper, we will assume that this has been done and $\rho_1(\theta)$ is real-valued.

The higher order parameters $\rho_k(\theta)$ for $k \geq 2$ also encode useful information.
If $\abs{\rho_k(\theta)} = 1$, then the phases of $\theta$ are all $k$-th roots of unity under some global shift.
Below, we explore the relationship between $\rho_1(\theta)$ and $\rho_2(\theta)$ for a stable state $\theta$.
These estimates will be of key importance in our proof.
Indeed, the \indef{amplification argument}, which is the core technique of our paper, derives a contradiction to these conditions from any unsynchronized stable state, which allows us to show global synchrony for a large class of expander and random graphs.

We assume now that $G$ is a connected graph and that a global shift of phases has been performed so that $\rho_1(\theta)$ is real and non-negative.
We further assume that the phases are all represented in the interval $(-\pi, \pi]$.

The following two lemmas are contained in the previous work of Kassabov, Strogatz, and Townsend \cite{Kassabov2021-rg}*{Lemma 1 and Lemma 2}.
We include the proof in the interest of keeping this paper self-contained.

\begin{lemma}
\label{lem:rho1-bound}
Let $\theta$ be a stable state on an $(n, \degm, \anorm)$-expander graph $G$.
Then
\begin{equation}
\label{eq:rho1-inequality}
\rho_1^2(\theta) \geq \frac{1 + \abs{\rho_2(\theta)}^2}{2} - 2 \anorm.
\end{equation}
In the case of $\degm$-regular graphs the above bound can be improved slightly to
\begin{equation}
\label{eq:rho1-inequality-regular}
\rho_1^2(\theta) \geq \frac{1 + \abs{\rho_2(\theta)}^2}{2} - \frac{3}{2} \anorm.
\end{equation}
\end{lemma}
\begin{proof}
Let $f \from V \to \CC$ be the function $f(x) = e^{i \theta_x}$.
Then
\begin{equation*}
	\inner{f}{Af} = (\degm/n) \inner{f}{Jf}  + \inner{f}{\Delta_A f}.
\end{equation*}
Since $\inner{f}{Jf} = n^2 \abs{\rho_1(\theta)}^2$ and
$\abs{\inner{f}{\Delta_A f}} \leq \norm{\Delta_A} \norm{f}^2$, we obtain that
\begin{equation*}
\abs[\Big]{\inner{f}{Af} - \degm n \abs{\rho_1(\theta)}^2 } \leq \anorm \degm n,
\end{equation*}
as $\norm{\Delta_A} \leq \anorm \degm$ and $\norm{f}^2 = n$.
Moreover, the equilibrium condition~\eqref{eq:equilibrium} gives
\begin{equation*}
	\inner{f}{Af} = \sum_{x,y \in V} A_{x,y} \cos(\theta_x -\theta_y).
\end{equation*}
Therefore, we get
\begin{equation*}
	\sum_{x,y \in V} A_{x,y} \cos(\theta_x -\theta_y) \leq  \degm n \abs{\rho_1(\theta)}^2 + \anorm \degm n.
\end{equation*}
We repeat the same arguments for $g, h \from V \to \CC$ given by $g(x) = e^{i 2\theta_x}$ and $h(x) = 1$,
\begin{equation*}
	\abs[\Big]{ \inner{g}{Ag} - \degm n \abs{\rho_2(\theta)}^2 } \leq \anorm \degm n,
	\quad \text{and} \quad
	\abs[\Big]{ \inner{h}{Ah} - \degm n } \leq \anorm \degm n,
\end{equation*}
which leads to
\begin{equation*}
	\sum_{x,y \in V} A_{x,y} \cos(2\theta_x -2\theta_y) \geq \degm n \abs{\rho_2(\theta)}^2 - \anorm \degm n,
	\quad \text{and} \quad
	\sum_{x,y \in V} A_{x,y} \geq \degm n  - \anorm \degm n.
\end{equation*}
We use the elementary identity $\cos(2\theta_x -2\theta_y)= 2\cos^2(\theta_x -\theta_y)-1$ to obtain that
\begin{align*}
	\sum_{x,y \in V} A_{x,y} \cos^2(\theta_x -\theta_y)
	&= \frac{1}{2} \sum_{x,y \in V} A_{x,y} \p[\big]{\cos(2\theta_x - 2\theta_y) + 1} \\
	&= \frac{1}{2}\sum_{x,y \in V} A_{x,y} \cos(2\theta_x -2\theta_y) + \frac{1}{2} \sum_{x,y \in V} A_{x,y} \\
	&\geq \frac{1}{2} \p[\big]{ \degm n \abs{\rho_2(\theta)}^2 - \anorm \degm n} + \frac{1}{2} \p[\big]{ \degm n  - \anorm \degm n} \\
	&= \degm n \p[\Big]{\frac{1 + \abs{\rho_2(\theta)}^2}{2}} - \anorm \degm n.
\end{align*}
Recall that the stability condition gives~\eqref{eq:cosine-condition}, which is equivalent to
\begin{equation*}
	\sum_{x,y \in V} A_{x,y} \cos(\theta_x - \theta_y) \geq \sum_{x,y \in V} A_{x,y} \cos^2(\theta_x - \theta_y).
\end{equation*}
Combining the bounds we obtained, we have
\begin{equation*}
	\degm n \rho_1^2(\theta) + \anorm \degm n
	\geq \degm n \p[\Big]{\frac{1 + \abs{\rho_2(\theta)}^2}{2}} - \anorm \degm n,
\end{equation*}
which simplifies to the claimed inequality~\eqref{eq:rho1-inequality}.

If the number of edges $\card{E}$ is equal to $\degm n/2$ (for example, if the graph is $\degm$-regular), then we use $\inner{h}{Ah} = \degm n$ instead of $\abs[\big]{ \inner{h}{Ah}  - \degm n } \leq \anorm \degm n$, which leads to the improved inequality~\eqref{eq:rho1-inequality-regular}.
\end{proof}

\begin{corollary}
\label{cor:rho1-not-zero}
Let $\theta$ be a stable state on an $(n, \degm, \anorm)$-expander graph $G$.
If $\anorm < 1/4$, then $\rho_1(\theta) > 0$.
\end{corollary}

For the next lemma, it will be convenient to introduce the following function
\begin{equation*}
	\Sin(\theta) \defined \begin{cases}
		\sin^2(\abs{\theta}) & \text{if } \abs{\theta} \leq \pi/2, \\
		1 &  \text{if } \pi/2 \leq \abs{\theta} \leq \pi.
	\end{cases}
\end{equation*}
We remark that $\Sin(\theta) \geq \sin^2(\abs{\theta}) = \abs{\sin(\theta)}^2 = \sin^2(\theta)$ and that $\Sin(\theta)$ is increasing on the interval $[0, \pi]$.

\begin{figure}[!ht]
\centering
\begin{tikzpicture}[scale=0.8]
    \begin{axis}[
        axis x line* = middle,
        axis y line* = middle,
        axis line style={-stealth},
        ymin=-0.1,
        ymax=1.1,
        xmin=-pi-0.5,
        xmax=pi+0.4,
        xtick = {-3.1415, -1.5707, 0, 1.5707, 3.1415},
        xticklabels = {$-\pi$, $-\pi/2$, $0$, $\pi/2$, $\pi$},
        every tick/.append style={black, very thick},
        x tick label style = {xshift={0.5em}},
        xlabel={$\theta$},
        xlabel style={at={(ticklabel* cs:1)}, anchor=south west},
        ytick = {-1, 0 , 1},
        yticklabels = {-1, 0 , 1},
        grid=major,
        grid style={dashed, gray},
        width = 8cm,
   ]
	\addplot [domain=-3.1415:-1.5707, black, ultra thick]
	{1};
	\addplot [domain=1.5707:3.1415, black, ultra thick]
	{1};
	\addplot [domain=-1.5707:1.5707, black, ultra thick, samples=200]
	{sin(abs(x*180/pi))^2};
\end{axis}
\end{tikzpicture}
\caption{The function $\Sin \from (-\pi, \pi] \to \RR$.}
\label{fig:function-sine}
\end{figure}

\begin{lemma}
\label{lem:rho1-sin}
Let $\theta$ be a stable state on an $(n, \degm, \anorm)$-expander graph $G$.
Then provided that $\rho_1(\theta) $ is real and non-negative, we have
\begin{equation*}
	n \anorm^2 \geq \rho_1^2(\theta) \sum_{x\in V}  \Sin(\theta_x).
\end{equation*}
\end{lemma}
\begin{proof}
Recall~\eqref{eq:ry-positivity},
where we observed that the stability conditions of $\theta$ implied that for every $y \in V$,
there is $r_y \geq 0$ such that $\sum_{x\in V} A_{x,y} e^{i \theta_x} = r_y e^{i \theta_y}$.
Recall that $A  = (\degm/n)J + \Delta_A$, thus
\begin{equation*}
	r_y = \degm \rho_1(\theta) e^{-i \theta_y} + \sum_{x \in V} (\Delta_A)_{x,y} e^{i (\theta_x - \theta_y)} \geq 0.
\end{equation*}
Taking the imaginary part on both sides, we obtain
\begin{equation*}
	\sum_{x \in V} (\Delta_A)_{x,y} \sin(\theta_x - \theta_y) = \degm \rho_1(\theta) \sin(\theta_y).
\end{equation*}
Now, by taking real parts, we obtain
\begin{equation*}
	\sum_{x \in V} (\Delta_A)_{x,y} \cos(\theta_x - \theta_y)
	\geq  - \degm \rho_1(\theta) \cos(\theta_y),
\end{equation*}
which implies that
\begin{equation*}
	\abs[\Big]{\sum_{x \in V} (\Delta_A)_{x,y} \cos(\theta_x - \theta_y)}
	\geq \degm \rho_1(\theta) \cos(\theta_y) \ind_{\set{ \pi/2 \leq \abs{\theta_y} \leq \pi }}.
\end{equation*}
Let $f \from V \to \CC$ be defined as $f(x) = e^{i (\theta_x - \theta_y)}$.
We have
\begin{align*}
	\abs[\big]{(\Delta_A f)_y}^2 &= \abs[\Big]{\sum_{x \in V} (\Delta_A)_{x,y}f(x)}^2 \\
	&= \p[\Big]{\sum_{x \in V} (\Delta_A)_{x,y}\cos(\theta_x - \theta_y)}^2
	+ \p[\Big]{\sum_{x \in V} (\Delta_A)_{x,y}\sin(\theta_x - \theta_y)}^2 \\
	&\geq \p[\big]{\degm \rho_1(\theta)}^2 \abs{\cos(\theta_y)}^2 \ind_{\set{ \pi/2 \leq \abs{\theta_y} \leq \pi }}
	+ \p[\big]{\degm \rho_1(\theta)}^2 \sin^2(\theta_y) \\
	&= \p[\big]{\degm \rho_1(\theta)}^2 \Sin(\theta_y).
\end{align*}
Therefore,
\begin{equation*}
	\norm[\big]{\Delta_A f}_ 2^2 = \sum_{y \in V} \abs[\big]{(\Delta_A f)_y}^2 \geq \p[\big]{\degm \rho_1(\theta)}^2 \sum_{y \in V} \Sin(\theta_y).
\end{equation*}
Together with $\norm{\Delta_A f}_ 2^2 \leq \norm{\Delta_A}^2 \norm{f}^2_ 2 \leq (\anorm \degm)^2 n$, we obtain the claimed inequality.
\end{proof}

\begin{corollary}
\label{cor:rho2-bound}
Let $\theta$ be a stable state on an $(n, \degm, \anorm)$-expander graph $G$.
Then provided that $\rho_1(\theta) $ is real and positive, we have
\begin{equation}
\label{eq:better-rho2-inequality}
	\abs{\rho_2(\theta)} \geq 1 - \frac{2 \anorm^2}{\rho_1^2(\theta)}.
\end{equation}
\end{corollary}
\begin{proof}
Since $\sin^2(\theta_x) \leq \Sin(\theta_x)$, we have
\begin{align*}
	\abs{\rho_2(\theta)}
	&\geq \Re\p[\big]{\rho_2(\theta)}
	= \frac{1}{n}\sum_{x \in V} \cos(2\theta_x) \\
	&= 1 - \frac{2}{n}\sum_{x \in V} \sin^2(\theta_x) \geq 1 - \frac{2}{n} \sum_{x \in V} \Sin(\theta_x).
\end{align*}
Applying \Cref{lem:rho1-sin} and that $\rho_1(\theta) \neq 0$, we obtain~\eqref{eq:better-rho2-inequality}.
\end{proof}

We now provide more tractable inequalities between the order parameters and $\anorm$.

\begin{lemma}
\label{lem:rhobounds}
Let $\theta$ be a stable state in an $(n, \degm, \anorm)$-expander graph.
If $\anorm \leq 1/5$, then
\begin{equation*}
	\rho_1^2(\theta) \geq  1 - 3 \anorm,
	\quad \text{ and } \quad
	\abs{\rho_2(\theta)} \geq 1 - 5\anorm^2.
\end{equation*}
\end{lemma}
\begin{proof}
To start, we observe that the condition that $\anorm \leq 1/5 < (\sqrt{2} - 1)/2$ implies that $\rho_1(\theta) \neq 0$ via \Cref{cor:rho1-not-zero} and that the lower bound~\eqref{eq:better-rho2-inequality} is always positive.
Indeed, if we had $1 - 2 \anorm^2 / \rho_1^2(\theta) \leq 0$, then by~\eqref{eq:rho1-inequality}, we would have
\begin{equation*}
    2 \anorm^2 \geq \rho_1^2(\theta) \geq \frac{1 + \abs{\rho_2(\theta)}^2}{2} - 2\anorm \geq \frac{1 - 4 \anorm}{2}.
\end{equation*}
This translates to $4 \anorm^2 + 4\anorm - 1 \geq 0$, which implies $\anorm \geq (\sqrt{2} - 1)/2$ when $\anorm \geq 0$.

Now, we use \eqref{eq:rho1-inequality} and \eqref{eq:better-rho2-inequality} to obtain
\begin{equation}
\label{eq:cubic-ineq}
    2\rho_1^2(\theta) + 4\anorm -1 \geq \p[\Big]{1 - \frac{2\anorm^2}{\rho_1^2(\theta)} }^2.
\end{equation}
We define the function
\begin{align*}
    f(x, \anorm) = 2x + 4\anorm-1 - \p[\big]{1 - 2\anorm^2 / x}^2,
\end{align*}
and note that \eqref{eq:cubic-ineq} is equivalent to $f(\rho_1^2(\theta), \anorm) \geq 0$.
To obtain said inequality, we can analyze $f(x, \anorm)$ under the restriction that $2 \anorm^2 / x < 1$ as we have shown it holds when $x = \rho_1^2(\theta)$.
Note that $f(x,\anorm) = 0$ is a cubic equation on $x$ with coefficients depending on $\anorm$ given by
\begin{equation}
\label{eq:cubic}
	x^3 + (2 \anorm - 1) x^2 + 2 \anorm^2 x - 2 \anorm^4 = 0.
\end{equation}
The discriminant of the cubic is
\begin{equation*}
	\Delta = - 4 \anorm^4 \p[\big]{27 \anorm^4 + 20\anorm^3 + 10\anorm^2 - 8\anorm + 1}.
\end{equation*}
In particular, $\Delta < 0$ for any $\anorm < 0.2055$.
As we assume that $\anorm < 1/5$, the equation $f(\xi, \anorm) = 0$ has a unique real root $\xi$.
Note that $f(x,\anorm) \to \infty$ for each fixed $\anorm$ when $x \to \infty$.
On the other hand, it can be easily checked via a quadratic equation that $f(0.431, \anorm) < 0$ for all $\anorm < 1/5$.
Hence the root $\xi$ of must satisfy $\xi \geq 0.431$ in this range.
Therefore, we have $\rho_1^2(\theta) \geq 0.431$, which leads to $\abs{\rho_2(\theta)} \geq 1 - 2 \anorm^2/0.431 \geq 1 - 5 \anorm^2$ via~\eqref{eq:better-rho2-inequality}.
Finally, when $\anorm \leq 1/5$, \eqref{eq:rho1-inequality} gives
\begin{equation*}
	\rho_1^2(\theta) \geq \frac{1 + (1 - 5\anorm^2)^2 - 4\anorm}{2} \geq 1 - 3\anorm. \qedhere
\end{equation*}
\end{proof}

Finally, we can put the previous inequality back in terms of sums of $\Sin(\theta_x)$.

\begin{corollary}
\label{cor:sin-bound}
Let $\theta$ be a stable state in an $(n, \degm, \anorm)$-expander graph.
If $\anorm \leq 1/5$, then
\begin{equation*}
	\frac{1}{n} \sum_{x\in V} \Sin(\theta_x) \leq \frac{5 \anorm^2}{2}.
\end{equation*}
\end{corollary}
\begin{proof}
By \Cref{lem:rhobounds}, we have $\rho_1^2(\theta) \geq  1 - 3 \anorm \geq 2/5$, as $\anorm < 1/5$.
The result follows from \Cref{lem:rho1-sin}.
\end{proof}

\begin{remark}
\label{rmk:regular}
The proof of \Cref{lem:rhobounds} can be easily adapted to show that every stable state $\theta$ on an $\degm$-regular $(n, \degm, \anorm)$-expander satisfies $\rho_1^2(\theta) \geq 1- 2\anorm$ when $\anorm \leq 0.1$.
The proof is identical, except we use \eqref{eq:rho1-inequality-regular} instead of \eqref{eq:rho1-inequality} and keep track of the new numbers, the relevant cubic being now $2x^3 + (3\anorm - 2)x^2 + 4\anorm^2 x + 4 \anorm^4$.
Similarly, the same conclusions of \Cref{lem:rhobounds} hold for a $\degm$-regular expander when $\anorm \leq 0.24585$.
Finally, the proof of \Cref{cor:sin-bound} gives that $\frac{1}{n} \sum_{x\in V} \Sin(\theta_x) \leq 5 \anorm^2/4$ for a $\degm$-regular $(n, \degm, \anorm)$-expander.
\end{remark}

\subsection{Kernel stability condition and the half-circle lemma}

Another useful consequence of the stability condition for a stable state $\theta$ is the following inequality that relates the state with an appropriate kernel.
Define the \indef{kernel} $K \from (-\pi,\pi] \times (-\pi,\pi] \to \RR$ as
\begin{equation}
\label{eq:kernel-definition}
	K(\alpha, \beta)
	= \sin\p[\big]{ \abs{\alpha} - \min\set{ \abs{\beta}, \pi/2 } }
	= \begin{cases}
		\sin(\abs{\alpha} - \abs{\beta}), & \abs{\beta} \leq \frac{\pi}{2}, \\
		-\cos(\alpha), & \abs{\beta} \geq \frac{\pi}{2}.
	\end{cases}
\end{equation}

\begin{lemma}[Kernel stability condition]
\label{lem:kernel-stability}
Let $\theta \from V \to (-\pi,\pi]$ be a stable state such that $\rho_1(\theta) \in [0,1]$.
Then, for any $y \in V$, we have
\begin{equation}
\label{eq:kernel-inequality}
	\sum_{x\in V} A_{x,y} K(\theta_x, \theta_y) \geq 0.
\end{equation}
\end{lemma}
\begin{proof}
Recall that the equilibrium condition~\eqref{eq:equilibrium-complex} gives that, for each $y \in V$, there is $r_y \in \RR$ with
\begin{equation*}
    \sum_{x\in V} A_{x,y} e^{i \theta_x} = r_y e^{i \theta_y},
\end{equation*}
and that the stability condition~\eqref{eq:ry-positivity} implies that $r_y \geq 0$.
We now divide our argument into three cases, depending on the value of $\theta_y$.
\begin{description}[leftmargin=0cm]
\item[Case 1 ($0 \leq \theta_y \leq \pi/2$)]
By noting that $\sin(\abs{\theta_x}) \geq  \sin (\theta_x)$, $\cos (\theta_y) \geq 0$ and $\cos(\abs{\theta_x}) = \cos(\theta_x)$, we obtain
\begin{align*}
	K(\theta_x, \theta_y)
	&= \sin(\abs{\theta_x} - \theta_y)
	= \sin(\abs{\theta_x}) \cos(\theta_y) - \cos(\abs{\theta_x}) \sin (\theta_y) \\
	&\geq \sin(\theta_x) \cos(\theta_y) - \cos(\theta_x) \sin(\theta_y)
	= \sin(\theta_x - \theta_y).
\end{align*}
Therefore,
\begin{equation*}
	\sum_{x\in V} A_{x,y} K(\theta_x, \theta_y)
	\geq \sum_{x\in V} A_{x,y} \sin(\theta_x - \theta_y)
	= \Im(r_y) = 0.
\end{equation*}

\item[Case 2 ($-\pi/2 \leq \theta_y \leq 0$)]
As before, observe that
\begin{align*}
	K(\theta_x, \theta_y)
	&= \sin(\abs{\theta_x} + \theta_y)
	= \sin(\abs{\theta_x}) \cos(\theta_y) + \cos (\abs{\theta_x}) \sin(\theta_y) \\
	&\geq -\sin(\theta_x) \cos(\theta_y) + \cos (\theta_x) \sin (\theta_y)
	= -\sin(\theta_x -\theta_y).
\end{align*}

\item[Case 3 ($\abs{\theta_y} \geq \pi/2$)]
In this case, $K(\theta_x, \theta_y) = -\cos(\theta_x)$
\begin{align*}
	\sum_{x\in V} A_{x,y} K(\theta_x, \theta_y)
	&= -\sum_{x\in V} A_{x,y} \cos(\theta_x) \\
	&= -\Re\p[\Big]{\sum_{x\in V} A_{x,y} e^{i \theta_x}}
	= - r_y \Re(e^{-i \theta_y}) \geq 0. \qedhere
\end{align*}
\end{description}
\end{proof}

We remark that such a kernel stability condition was already present in Lemma 7 of Kassabov, Strogatz, and Townsend~\cite{Kassabov2021-rg}, though they used a slightly different kernel.

To end this section, we present a well-known sufficient condition to guarantee a stable state is the all-in-phase equilibrium.
For completeness, we use the kernel stability condition to provide an alternative proof of this standard result~\cites{Jadbabaie2004-lf,Ochab2010-ps,Taylor2012-jc,Dorfler2014-kx}.

\begin{lemma}[Half-circle lemma]
\label{lem:half-circle}
Let $\theta$ be a stable state on a connected graph $G$ with $\rho_1(\theta) \in [0,1]$.
If $\abs{\theta_x} < \pi/2$ for every $x \in V$, then $\theta_x = 0$ for every $x \in V$.
\end{lemma}
\begin{proof}
Since $\theta$ is a stable state and $\abs{\theta_x} < \pi/2$ for every $x \in V$, we have $K(\theta_x, \theta_z) = \sin(\abs{\theta_x} - \abs{\theta_z})$.
Let $\abs{\theta_z} = \max_{x \in V} \abs{\theta_x}$, so for every $x \in V$, we have $-\pi < \abs{\theta_x} - \abs{\theta_z} \leq 0$, which implies $\sin(\abs{\theta_x} - \abs{\theta_z}) < 0$ unless $\abs{\theta_x} = \abs{\theta_z}$.
Now, from \Cref{lem:kernel-stability} we have that
\begin{equation*}
	\sum_{x \in V} A_{x,z} K(\theta_x, \theta_z)
	= \sum_{x \in V} A_{x,z} \sin(\abs{\theta_x} - \abs{\theta_z})
	\geq 0.
\end{equation*}
Thus $\abs{\theta_x} = \abs{\theta_z}$ for every $x$ neighbor of $z$.
As $G$ is connected, repeating this argument leads to $\abs{\theta_x} = \abs{\theta_z}$ for all $x \in V$.
Therefore, there are $c \in [0, \pi/2)$ and a set $S \subset V$ with $\theta_x = c$ for $x \in V$ and $\theta_y = -c$ for $y \notin V$.
Summing over $y \in \setcomp{S}$ on the equilibrium condition~\eqref{eq:equilibrium}, we obtain
\begin{equation*}
	\sum_{y \in \setcomp{S}} \sum_{x \in V} A_{x,y} \sin(\theta_x - \theta_y) = \sum_{y \in \setcomp{S}} \sum_{x \in S} A_{x,y} \sin(2c) = 0.
\end{equation*}
As $G$ is connected, there is an edge from $S$ to $\setcomp{S}$.
Finally, as $0 \leq 2c < \pi$, we have $\sin(2c) > 0$ unless $c = 0$.
Therefore, $\theta_x = c = 0$ for all $x \in V$.
\end{proof}


\section{Preliminaries on spectral graph theory}
\label{sec:edgebounds}

Let $G = (V,E)$ be a graph with $n$ vertices and adjacency matrix $A$.
The degree of a vertex $x \in V$ is denoted by $d(x) \defined \sum_{y \in V} A_{x,y}$, so edges forming self-loops are counted once.\footnote{Self-loops do not affect Kuramoto dynamics as diagonal entries cancel in $\grad \cE_G$.}
The minimum and maximum degree of vertices in $G$ is denoted $\degmin(G)$ and $\degmax(G)$, respectively.
Given subsets $X, Y \subseteq V$, we write $e(X,Y) \defined \sum_{x \in X}\sum_{y \in Y} A_{x,y}$.

Recall that $D$ is the \indef{diagonal degree matrix} with $D_{x,x} = \deg(x)$ and $L$ is the Laplacian matrix with $L = D - A$.
We start by collecting a few facts related to our notion of spectral expanders.
We remark that these facts are standard in the literature for regular graphs and we adapted the proofs for the non-regular case (see, for example, Chung's book~\cite{Chung1997-se}).

\begin{lemma}
\label{lem:minmaxdegree}
If $G$ is an $(n, \degm, \anorm, \lnorm^-, \lnorm^+)$-expander, then
\begin{equation*}
    (1 + \lnorm^-)\degm - \frac{\degm}{n} \leq \degmin(G) \leq \degmax(G) \leq (1 + \lnorm^+)\degm + 1 - \frac{\degm}{n}.
\end{equation*}
\end{lemma}
\begin{proof}
A positive semidefinite matrix has non-negative diagonal entries.
Thus, the condition $\lnorm^- \degm I \preceq \Delta_L \preceq \lnorm^+ \degm I$ implies that $\lnorm^- \degm \leq (\Delta_L)_{x,x} \leq \lnorm^+ \degm$.
Therefore,
\begin{equation*}
    (1 + \lnorm^-) \degm \leq d(x) - A_{x,x} + \degm/n \leq (1 + \lnorm^+) \degm.
\end{equation*}
The result follows as $A_{x,x} \in \set{0,1}$.
\end{proof}

The following proposition provides a partial converse of the previous lemma.

\begin{proposition}
\label{prop:degree-spectral}
Let $G$ be a graph with $n$ vertices.
If $\norm{\Delta_A} \leq \anorm \degm$ and
\begin{equation*}
    (1 + \lnorm^- + \anorm) \degm  \leq \degmin(G) \leq \degmax(G) \leq (1 + \lnorm^+ - \anorm) \degm,
\end{equation*}
then $G$ is an $(n, \degm, \anorm, \lnorm^-, \lnorm^+)$-expander.
\end{proposition}
\begin{proof}
The condition on the degree implies, as $D$ is a diagonal matrix, that
\begin{align*}
    (1 + \lnorm^-)\degm I + \anorm \degm I
    \preceq D \preceq
    (1 + \lnorm^+)\degm I - \anorm \degm I.
\end{align*}
Since $\norm{\Delta_A} \leq \anorm \degm$, we have $-\anorm \degm I \preceq \Delta_A \preceq \anorm \degm I$, thus
\begin{align*}
    (1 + \lnorm^-)\degm I  + \Delta_A
    \preceq D \preceq
    (1 + \lnorm^+)\degm I + \Delta_A,
\end{align*}
which translates to $\lnorm^- \degm I \preceq \Delta_D - \Delta_A  = \Delta_L \preceq \lnorm^+ \degm I$ as claimed.
\end{proof}

One can see that $\lnorm^- > -1$ is equivalent to the graph being connected.
Indeed, $G$ is connected if and only if $L$ has a single eigenvalue equal to $0$.
The operator $\degm I - (\degm/n) J$ has a single eigenvalue $0$, and an eigenvalue $\degm$ with multiplicity $n-1$.
Thus the second smallest eigenvalue of $L$ is bounded below by $\degm +\lnorm^- \degm > 0$ when $\lnorm^- > -1$.

Expander graphs are known to satisfy so-called \emph{expander mixing lemmas}, which essentially guarantee that the number of edges between two sets of vertices (of sufficiently large size) is comparable to the expected number of edges that could be computed by taking into account the graph's average degree, and the size of the two sets.
We now proceed by establishing a number of such results for our notion of expander graph.
This will consist in estimating the number of edges across several types of cuts in an $(n, \degm, \anorm, \lnorm^-, \lnorm^+)$-expander.

\begin{lemma}
\label{lem:eXX}
If $G$ is an $(n, \degm, \anorm)$-expander, then for every $X \subseteq V$, we have
\begin{equation*}
    -\anorm \degm \card{X} \leq e(X,X) - (\degm/n)\card{X}^2 \leq
    \anorm \degm \card{X}.
\end{equation*}
\end{lemma}
\begin{proof}
Let $\ind_X$ be the indicator function of $X \subseteq V$.
We have $\inner{\ind_X}{A \ind_X} = e(X,X)$, $\inner{\ind_X}{\ind_X} = \card{X}$ and $\inner{\ind_X}{J \ind_X} = \card{X}^2$.
Therefore,
\begin{equation*}
    \abs[\big]{e(X,X) - (\degm/n) \card{X}^2}
    = \abs[\big]{\inner{\ind_X}{\Delta_A \ind_X}}
    \leq \norm{\Delta_A} \inner{\ind_X}{\ind_X}
    \leq \anorm \degm \card{X}. \qedhere
\end{equation*}
\end{proof}

\begin{lemma}
\label{lem:eXXc}
If $G$ is an $(n, \degm, \anorm, \lnorm^-, \lnorm^+)$-expander, then for every $X \subseteq V$, we have
\begin{equation*}
    \lnorm^- \p[\Big]{\frac{\degm}{n}} \card{X}\card{\setcomp{X}} \leq e(X,\setcomp{X}) - \p[\Big]{\frac{\degm}{n}} \card{X}\card{\setcomp{X}} \leq
    \lnorm^+ \p[\Big]{\frac{\degm}{n}} \card{X}\card{\setcomp{X}}.
\end{equation*}
\end{lemma}
\begin{proof}
Consider the function $f = \card{\setcomp{X}} \ind_X - \card{X}\ind_{\setcomp{X}}$.
One can easily verify that $\inner{f}{f} = n\card{X}\card{\setcomp{X}}$, $\inner{f}{Lf} = n^2 e(X,\setcomp{X})$ and $\inner{f}{Jf} = 0$.
Therefore,
\begin{equation*}
	\inner{f}{\Delta_L f}
	= \p[\big]{e(X,\setcomp{X}) - (\degm/n)\card{X}\card{\setcomp{X}}}n^2.
\end{equation*}

Since $\Delta_L - \lnorm^- \degm I$ is positive semidefinite, $\inner{f}{(\Delta_L - \lnorm^- \degm I)f} \geq 0$ and then
\begin{equation*}
    e(X,\setcomp{X}) - (\degm/n) \card{X}\card{\setcomp{X}}
    \geq \lnorm^- (\degm/n) \card{X} \card{\setcomp{X}},
\end{equation*}
which is the lower bound.
The upper bound follows similarly from $\lnorm^+ \degm I - \Delta_L \succeq 0$.
\end{proof}

\begin{remark}
The above lemma essentially says that the isometric constant, also known as the Cheeger constant, of an $(n, \degm, \anorm, \lnorm^-, \lnorm^+)$-expander is bounded below by $(1 + \lnorm^-)\degm/2$, which is another reason why $\lnorm^- > -1$ is equivalent to connectivity.
\end{remark}

\begin{lemma}
\label{lem:eXV}
If $G$ is an $(n, \degm, \anorm, \lnorm^-, \lnorm^+)$-expander, then for every subset $X \subseteq V$ with density $\rho = \card{X} /n$, we have
\begin{equation*}
	\p[\big]{1 + \lnorm^-(1 - \rho) - \anorm} \degm \card{X}  \leq e(X,V) \leq \p[\big]{1 + \lnorm^+(1 - \rho) + \anorm} \degm \card{X}.
\end{equation*}
\end{lemma}
\begin{proof}
From \Cref{lem:eXX} and \Cref{lem:eXXc}, we write
\begin{align*}
	e(X, V)
	&= e(X,X) + e(X,\setcomp{X}) \\
	&\leq (\degm/n) \card{X}^2 + \anorm \degm \card{X} + (1 + \lnorm^+) (\degm/n) \card{X}\card{\setcomp{X}} \\
	&= \p[\big]{1 + \lnorm^+\card{\setcomp{X}}/n + \anorm} \degm\card{X} \\
	&= \p[\big]{1 + \lnorm^+(1 - \rho) + \anorm} \degm \card{X}.
\end{align*}
The lower bound is analogous.
\end{proof}

In particular, this lemma implies that $\degmax(G) \leq (1 + \lnorm^+ + \anorm)\degm$.
We need a slightly more elaborate bound for the number of edges $e(X,Y)$ between two disjoint sets in the graph.
Indeed, we would like to have an estimate of the form $e(X,Y) \geq \eps (\degm/n) \card{X} \card{Y}$.
However, such a bound does not hold unless we make additional assumptions on the sets $X$ and $Y$.

\begin{lemma}
\label{lem:eXYc}
Let $G$ be an $(n, \degm, \anorm, \lnorm^-, \lnorm^+)$-expander and $\eps > 0$.
If $X, Y \subseteq V$ satisfy $X \subseteq Y$, $\card{Y} \leq n/2$ and $(1 + \delta) \card{X} \geq \card{Y}$, where $\delta = \eps/(\lnorm^+ - \lnorm^-)$, then
\begin{equation*}
    (1 + \lnorm^- - \eps)  \p[\Big]{\frac{\degm}{n}} \card{X} \card{\setcomp{Y}}
    \leq e(X, \setcomp{Y}) \leq
    (1 + \lnorm^+ + \eps)  \p[\Big]{\frac{\degm}{n}} \card{X} \card{\setcomp{Y}}.
\end{equation*}
\end{lemma}
\begin{proof}
We start with the following identity,
\begin{align*}
    e(X,\setcomp{Y}) = \frac{e(X,\setcomp{X}) + e(Y, \setcomp{Y}) - e(Y\setminus X,\setcomp{(Y \setminus X)} )}{2}.
\end{align*}
By \Cref{lem:eXXc}, we obtain a lower/upper bound for
$e(X, \setcomp{Y}) - (\degm/n) \card{X} \card{\setcomp{Y}}$ by using lower/upper bounds for $e(Z, \setcomp{Z})$ for suitable sets $Z$.
This gives
\begin{align*}
    e(X, \setcomp{Y}) - (\degm/n) \card{X} \card{\setcomp{Y}}
    &\geq \frac{\degm}{2n} \p[\Big]{ \lnorm^- \card{X} \card{\setcomp{X}} +
    \lnorm^- \card{Y} \card{\setcomp{Y}} -
    \lnorm^+ \card{Y\setminus X } \card{\setcomp{(Y\setminus X)}} } \\
    &= \frac{\degm}{2n} \p[\Big]{ 2 \lnorm^- \card{X} \card{\setcomp{Y}} + (\lnorm^- - \lnorm^+) \card{Y\setminus X } \card{\setcomp{(Y\setminus X)}} },
\end{align*}
and
\begin{align*}
    e(X, \setcomp{Y}) - (\degm/n) \card{X} \card{\setcomp{Y}}
    &\leq \frac{\degm}{2n} \p[\Big]{ \lnorm^+ \card{X} \card{\setcomp{X}} +
    \lnorm^+ \card{Y} \card{\setcomp{Y}} -
    \lnorm^- \card{Y\setminus X } \card{\setcomp{(Y\setminus X)}} } \\
    &= \frac{\degm}{2n} \p[\Big]{ 2 \lnorm^+ \card{X} \card{\setcomp{Y}} + (\lnorm^+ - \lnorm^-) \card{Y\setminus X } \card{\setcomp{(Y\setminus X)}} }.
\end{align*}

Adding the assumptions that $\card{Y \setminus X } \leq \delta \card{X}$ leads to
\begin{align*}
    e(X, \setcomp{Y}) - (\degm/n) \card{X} \card{\setcomp{Y}}
    &\geq \lnorm^- (\degm/n) \card{X} \card{\setcomp{Y}} + \p[\big]{(\lnorm^- - \lnorm^+)/2} \delta \degm \card{X} \\
    &\geq \p[\big]{ \lnorm^- + \delta(\lnorm^- - \lnorm^+)} (\degm/n) \card{X} \card{\setcomp{Y}},
\end{align*}
where in the last step we used that $\card{\setcomp{Y}} \geq n/2$.
Similarly,
\begin{align*}
    e(X, \setcomp{Y}) - (\degm/n) \card{X} \card{\setcomp{Y}}
    &\leq \lnorm^+ (\degm/n) \card{X} \card{\setcomp{Y}} + \p[\big]{(\lnorm^+ - \lnorm^-)/2} \delta \degm \card{X} \\
    &\leq \p[\big]{ \lnorm^+ + \delta(\lnorm^+ - \lnorm^-)} (\degm/n) \card{X} \card{\setcomp{Y}}.
\end{align*}
As $\delta = \eps / (\lnorm^+ - \lnorm^-)$, we are done.
\end{proof}

Later, it is useful to have a slightly modified version of \Cref{lem:eXYc} in the following form.

\begin{lemma}
\label{lem:eXYc-smallX}
Let $G$ be an $(n, \degm, \anorm, \lnorm^-, \lnorm^+)$-expander and $\eps, \rho > 0$.
If $X, Y \subseteq V$ satisfy  $X \subseteq Y$, $\card{Y} \leq n/2$, $(1 + \delta) \card{X} \geq \card{Y}$ and $\card{X} \leq \rho \anorm n$, where $\delta = \eps/(\lnorm^+ - \lnorm^-)$, then
\begin{equation*}
    e(X,\setcomp{Y}) \geq \frac{(1 + \lnorm^- - \eps)}{2(1 + \rho)\anorm} \, e(X,X).
\end{equation*}
\end{lemma}
\begin{proof}
From \Cref{lem:eXX} and the fact that $\card{X} \leq \rho \anorm n$, we have
\begin{align*}
    e(X,X) \leq (\degm/n) \card{X}^2 + \anorm \degm \card{X}
    = \p[\big]{ \card{X}/n + \anorm } \degm \card{X}
    \leq (1 + \rho) \anorm \degm \card{X}.
\end{align*}
Using the lower bound in \Cref{lem:eXYc}, we obtain
\begin{align*}
    e(X, \setcomp{Y})
    &\geq (1 + \lnorm^- - \eps)  (\degm/n) \card{X} \card{\setcomp{Y}} \\
    &\geq \frac{(1 + \lnorm^- - \eps) \card{\setcomp{Y}}}{(1 + \rho)\anorm n} \, e(X,X)
    \geq \frac{(1 + \lnorm^- - \eps)}{2(1 + \rho)\anorm} \, e(X,X),
\end{align*}
where in the last step, we used that $\card{\setcomp{Y}} \geq n/2$.
\end{proof}

\begin{remark}
\label{rem:strict}
We note that the inequalities in the conclusions of \Cref{lem:eXYc} and \Cref{lem:eXYc-smallX} can be made into strict inequalities by making some of the inequalities in the assumption into strict inequalities.
\end{remark}


\section{The amplification argument}
\label{sec:amplification}

In this section we prove \Cref{thm:amplification-main} using the amplification argument, which we will briefly outline.
Our goal is to show that in a sufficiently expanding graph, the only stable state is the fully synchronized one.
To do so, we start with a stable state $\theta$ which we assume is not fully synchronized and we derive a contradiction from the stability conditions developed in \Cref{sec:kuramoto-graph} and the expansion properties from \Cref{sec:edgebounds}.

Given a stable state $\theta$, we will keep track of the set of vertices whose phases are in some special arcs.
For any angle $0 \leq \psi \leq \pi$ and state $\theta$, the set $\cC_\psi(\theta)$ is defined as
\begin{equation}
\label{eq:Cpsi}
	\cC_\psi = \cC_\psi(\theta) \defined \set[\big]{x \in V \st \cos(\theta_x) \leq \cos(\psi)}.
\end{equation}
Alternatively, we have $\cC_\psi = \set[\big]{x \in V \st \abs{\theta_x} \geq \abs{\psi}}$.
Thanks to the half-circle lemma (recall \Cref{lem:half-circle}), our task boils down to showing that the set $\cC_{\pi/2}$ is empty for any stable state $\theta$.

In a nutshell, our main argument proceeds by contradiction: we assume that there is a non-trivial stable state and then, by \Cref{lem:half-circle}, the set $\cC_{\pi/2}$ is not empty.
This implies that $\card{\cC_{\pi/2}} \defines c_0 \geq 1$.
We inductively construct a sequence of angles
\begin{equation*}
	\pi/2 = \beta_0 > \beta_1 > \beta_2 > \dotsb > \beta_M \geq 0,
\end{equation*}
with the corresponding lower bounds $\card{\cC_{\beta_s}} \geq c_s$ for $1 \leq s \leq M$, with
\begin{equation*}
	c_M > \dotsb > c_2 > c_1 > c_0 = \card{\cC_{\pi/2}} \geq 0.
\end{equation*}
For the convenience of notation, we take $\beta_{-1} = \pi$, $c_{-1} = 0$ and $\beta_{M+1} = 0$, $c_{M+1} = n$,
hence $\card{\cC_{\beta_s}} \geq c_{s}$ holds for all $-1 \leq s \leq M + 1$.

\begin{figure}[!ht]
\centering
\begin{tikzpicture}[scale=1.0]
	\begin{axis}[
		xmin=-5.8,xmax=5.8,ymin=-0.4,ymax=5.5,
      axis lines = none,
      ticks =none,
      unit vector ratio*=1 1 1,
   ]
   \addplot [thick, smooth, domain=(0:2*pi)] ({4*cos(deg(x))},{4*sin(deg(x))})
   node[left,pos=0.5]{$\beta_{-1}$}
   node[above,pos=0.25]{$\beta_0$}
   node[above right,pos=0.24]{$\beta_1$}
   node[above right,pos=0.19]{$\ddots$}
   node[above right,pos=0.125]{$\beta_{M}$}
   node[right,pos=0.0]{$\beta_{M+1}$};
   \addplot [only marks, mark=*, mark options={black, solid}, samples at={0, 0.7853, 0.9, 1.0, 1.1, 1.2, 1.3, 1.44, 1.5707, 3.1416},mark size=.05cm] ({4*cos(deg(x))}, {4*sin(deg(x))});
\end{axis}
\end{tikzpicture}
\caption{Angles $\beta_{-1} \geq \beta_0 \geq \beta_1 \geq \dotsb \geq \beta_M \geq \beta_{M+1}$ in the amplification argument.}
\label{fig:angles}
\end{figure}

The contradiction will be reached by obtaining, simultaneously, that $c_M$ is very large (almost equal to $n$) and the angle $\beta_M \geq \eps > 0$.
Thus, the set $\cC_{\beta_M}$ contains a big portion of the vertices in the graph, which contradicts the bounds on $\rho_1(\theta)$ and $\rho_2(\theta)$ from \Cref{lem:rhobounds}.
The expansion properties of the graph will be crucial to obtain a good lower bound on $\card{\cC_{\beta_M}}$ while keeping $\beta_M$ bounded away from 0.

We now derive quantitative inequalities related to this idea.

\begin{lemma}
\label{lem:amplification}
Under the above notation, the following holds:
\begin{equation*}
	\sum_{x \in V} \Sin(\theta_x)
	\geq c_M \sin^2(\beta_M).
\end{equation*}
\end{lemma}
\begin{proof}
Since $H(\theta)$ is a non-negative non-increasing function, we have 
\begin{align*}
    \sum_{x\in V} \Sin(\theta_x)
    &\geq \sum_{x\in \cC_{\beta_M}} \Sin(\theta_x)
    \geq \sum_{x\in \cC_{\beta_M}}\sin^2(\theta_x)\\
    &\geq \sum_{x\in \cC_{\beta_M}}\sin^2(\beta_M)
    = \abs{\cC_{\beta_M}} \sin^2(\beta_M) \geq c_M\sin^2(\beta_M). \qedhere
\end{align*}
\end{proof}

By \Cref{cor:sin-bound} and \Cref{lem:amplification}, we have
\begin{equation*}
	 \frac{5 \anorm^2 n}{2}
	 \geq \sum_{x \in V} \Sin(\theta_x)
	 \geq c_M \sin^2(\beta_M).
\end{equation*}
This means that we reach a contradiction if
\begin{equation}
\label{eq:main-contradiction}
	c_M \sin^2(\beta_M) > \frac{5 \anorm^2 n}{2}.
\end{equation}
Therefore, \Cref{thm:amplification-main} will be proved if we construct a sequence of angles such that the last angle satisfies~\eqref{eq:main-contradiction}.
Clearly, we need to control the increments $\beta_{k} - \beta_{k-1}$ and the corresponding lower bounds $c_k$.
We state the two key lemmas to control them.

\begin{lemma}
\label{lem:amplification-simple}
Let $\theta$ be a stable state on an $(n, \degm, \anorm, \lnorm^-, \lnorm^+)$-expander and consider $0 < \eps < 1 + \lnorm^-$ and $\rho > 0$.
If the angles $0 < \gamma < \beta \leq \pi/2$ satisfy
\begin{equation*}
\sin(\beta - \gamma) \geq \frac{2(1 + \rho)\anorm}{(1 + \lnorm^- - \eps)},
\end{equation*}
then the following holds:
\begin{equation*}
\card{\cC_\gamma} \geq \min \set[\Big]{
\p[\Big]{1 + \frac{\eps}{\lnorm^+ - \lnorm^-}} \card{\cC_\beta}
,\; \rho \anorm n
,\; n/2 }.
\end{equation*}
\end{lemma}

\begin{lemma}
\label{lem:amplification-cancelation}
Let $\theta$ be a stable state on an $(n, \degm, \anorm, \lnorm^-, \lnorm^+)$-expander and consider $0 < \eps < 1 + \lnorm^-$.
If the angles $0 < \gamma < \beta \leq \pi/2$ satisfy
\begin{equation*}
	\sin(\beta - \gamma) \geq \frac{2(1+ \lnorm^+ + \anorm)\card{\cC_{\pi/2}}}{(1 + \lnorm^- - \eps) \card{\cC_\beta}},
\end{equation*}
then the following holds:
\begin{equation*}
	\card{\cC_\gamma} \geq \min \set[\Big]{
	\p[\Big]{1 + \frac{\eps}{\lnorm^+ - \lnorm^-}} \card{\cC_\beta}
	,\;  n/2 }.
\end{equation*}
\end{lemma}

The main takeaway from these two lemmas is that under suitable conditions, we can prove that $\card{\cC_{\gamma - \eta}} \geq (1 + \delta)\card{\cC_{\gamma}}$ for some $\delta > 0$, given that $\eta > 0$ is not so small.
This result will allow us to amplify our initial lower bound of $\card{\cC_{\pi/2}} \geq c_0$ to a lower bound of $\card{\cC_{\beta_M}} > \p[\big]{5 \anorm^2 / 2 \sin^2(\beta_M)} n$.

\begin{proof}[Proof of \Cref{lem:amplification-simple}]
By \Cref{lem:kernel-stability}, for any $y \in V$, we know that
\begin{equation*}
	\sum_{x \in V} A_{x,y}K(\theta_x, \theta_y) \geq 0.
\end{equation*}
We sum the above inequality over all $y \in \cC_\beta$ to obtain that
\begin{equation*}
	\sum_{y \in \cC_{\beta}}\sum_{x \in V} A_{x,y}K(\theta_x, \theta_y) \geq 0.
\end{equation*}
We split the sum over $x \in V$ into three parts as follows,
\begin{equation}
\label{eq:kernel-sum-decomposition}
	\sum_{\substack{x \in \cC_\beta \\ y \in \cC_\beta}} A_{x,y} K(\theta_x, \theta_y)
	+ \sum_{\substack{x \in \cC_\gamma \setminus \cC_\beta \\ y \in \cC_\beta}} A_{x,y}K(\theta_x, \theta_y)
	+ \sum_{\substack{x \in \cC_\gamma^\comp \\ y \in \cC_\beta}} A_{x,y}K(\theta_x, \theta_y)
	\geq 0.
\end{equation}
Recall the definition of the kernel~\eqref{eq:kernel-definition}.
Since $y \in \cC_\beta$ and $0 < \gamma < \beta \leq \pi/2$, after some case checking, we see that the following inequality holds:
\begin{equation*}
	K(\theta_x, \theta_y)
	\leq \begin{cases}
		1, & x \in \cC_\beta, \\
		0, & x \in \cC_\gamma \setminus \cC_\beta, \\
		\sin(\gamma -\beta ), & x \in \cC_\gamma^\comp.
	\end{cases}
\end{equation*}

Therefore, from~\eqref{eq:kernel-sum-decomposition}, we have
\begin{equation*}
	\sum_{\substack{x \in \cC_\beta \\ y \in \cC_\beta}} A_{x,y} \cdot 1
	+ \sum_{\substack{x \in \cC_\gamma \setminus \cC_\beta \\ y \in \cC_\beta}} A_{x,y} \cdot 0
	+ \sum_{\substack{x \in \cC_\gamma^\comp \\ y \in \cC_\beta}} A_{x,y} \cdot \sin(\gamma - \beta) \geq 0,
\end{equation*}
which simplifies to
\begin{equation}
\label{eq:kernel-edge-count}
	e(\cC_\beta, \cC_\beta) + \sin(\gamma -\beta)  e(\cC_\beta, \cC_\gamma^\comp) \geq 0.
\end{equation}

Now we can use the fact that we have an $(n, \degm, \anorm, \lnorm^-, \lnorm^+)$-expander graph and use expansion to give a lower bound on $e(\cC_\beta, \cC_\gamma^\comp)$.
Indeed, by \Cref{lem:eXYc-smallX} (and \Cref{rem:strict}), under the assumption that $\card{\cC_\gamma} < n/2$, that $\card{\cC_\gamma} < \p[\big]{1 + \eps/(\lnorm^+ - \lnorm^-)} \card{\cC_\beta}$ and that $\card{\cC_\beta} < \rho \anorm n$, we have
\begin{equation*}
    e(\cC_\beta, \cC_\gamma^\comp) > \frac{(1 + \lnorm^- - \eps)}{2(1+ \rho)\anorm} \, e(\cC_\beta, \cC_\beta).
\end{equation*}
Thus, under these assumptions and the lower bound on $\sin(\beta - \gamma)$, \eqref{eq:kernel-edge-count} gives
\begin{equation*}
	e(\cC_\beta, \cC_\beta) \geq \sin(\beta - \gamma) e(\cC_\beta, \cC_\gamma^\comp) > e(\cC_\beta, \cC_\beta),
\end{equation*}
a contradiction.
Therefore, the aforementioned assumptions cannot all hold,
which gives
\begin{equation*}
	\card{\cC_\gamma} \geq \min \set[\Big]{
	\p[\Big]{1 + \frac{\eps}{\lnorm^+ - \lnorm^-}} \card{\cC_\beta}
	,\; \rho \anorm n
	,\; n/2 },
\end{equation*}
where in the case that $\card{\cC_\beta} < \rho \anorm n$, we used that $\cC_\beta \subseteq \cC_\gamma$.
\end{proof}

\begin{proof}[Proof of \Cref{lem:amplification-cancelation}]
The proof of this lemma is similar to the proof of \Cref{lem:amplification-simple}.
The main difference is that in the first sum in~\eqref{eq:kernel-sum-decomposition} we use the fact that the kernel $K$ is almost anti-symmetric.
That is, the following inequality holds:
\begin{equation*}
	K(\alpha, \beta) + K(\beta, \alpha) \leq
	\begin{cases}
		0, & \abs{\alpha}, \abs{\beta} < \pi/2, \\
		1, & \abs{\alpha} < \pi/2, \abs{\beta} \geq \pi/2,\\
		1, & \abs{\alpha} \geq \pi/2, \abs{\beta} < \pi/2,\\
		2, & \abs{\alpha}, \abs{\beta} \geq \pi/2,
\end{cases}
\end{equation*}
or written compactly, writing with indicator functions, we have
\begin{equation*}
	K(\alpha, \beta) + K(\beta, \alpha) \leq \ind_{\abs{\alpha} \geq \pi/2} + \ind_{\abs{\beta} \geq \pi/2}.
\end{equation*}

Together with the symmetry of the adjacency matrix, this almost anti-symmetry of the kernel gives that
\begin{align*}
	\sum_{x,y \in \cC_\beta} A_{x,y} K(\theta_x, \theta_y)
	&= \frac{1}{2} \sum_{x,y \in \cC_\beta} A_{x,y} \p[\big]{K(\theta_x, \theta_y) + K(\theta_y, \theta_x)} \\
	&\leq \frac{1}{2} \sum_{x,y \in \cC_\beta} A_{x,y} \p[\big]{\ind_{\abs{\theta_x} \geq \pi/2} + \ind_{\abs{\theta_y} \geq \pi/2}}
	= e(\cC_{\pi/2}, \cC_\beta),
\end{align*}
where in the last inequality we used the fact that $\cC_{\pi/2} \subseteq \cC_\beta$.
Clearly, $e(\cC_{\pi/2}, \cC_\beta)$ is at most $e(\cC_{\pi/2}, V)$, so by \Cref{lem:eXV}, we have
\begin{equation*}
	e(\cC_{\pi/2}, \cC_\beta) \leq e(\cC_{\pi/2}, V)
	\leq (1 + \lnorm^+ + \anorm) \card{\cC_{\pi/2}}.
\end{equation*}

\noindent Therefore, from~\eqref{eq:kernel-sum-decomposition}, we have
\begin{equation*}
	e(\cC_{\pi/2}, \cC_\beta) \geq \sin(\beta - \gamma)  e(\cC_\beta, \cC_\gamma^\comp).
\end{equation*}

\noindent Under the assumptions that $\card{\cC_\gamma} < n/2$ and $\card{\cC_\gamma} < \p[\big]{1 + \eps/(\lnorm^+ - \lnorm^-)} \card{\cC_\beta}$, \Cref{lem:eXYc} (and \Cref{rem:strict}) gives that
\begin{equation*}
	e(\cC_\beta, \cC_\gamma^\comp) > \p[\big]{1 + \lnorm^- - \eps} (\degm/n) \card{\cC_\beta} \card{\cC_\gamma^\comp}.
\end{equation*}
By the hypothesis on $\sin(\beta - \gamma)$, we have
\begin{align*}
	(1 + \lnorm^+ + \anorm)\degm \card{\cC_{\pi/2}}
	&\geq e(\cC_\beta, \cC_{\pi/2})
	\geq \sin(\beta - \gamma) e(\cC_\beta, \cC_\gamma^\comp) \\
	&> 2(1+ \lnorm^+ + \anorm) (\degm/n) \card{\cC_{\pi/2}} \card{\cC_\gamma^\comp}
	\geq (1 + \lnorm^+ + \anorm) \degm \card{\cC_{\pi/2}},
\end{align*}
where in the last step we used that $\card{\cC_\gamma^\comp} \geq n/2$.
This is a contradiction to our assumptions on $\cC_\gamma$, therefore
\begin{equation*}
	\card{\cC_\gamma} \geq \min \set[\Big]{
	\p[\Big]{1 + \frac{\eps}{\lnorm^+ - \lnorm^-}} \card{\cC_\beta}
	,\;  n/2 }. \qedhere
\end{equation*}
\end{proof}

We are finally ready to prove the main result of this section.

\begin{proof}[Proof of \Cref{thm:amplification-main}]
Let $G$ be an $(n, \degm, \anorm, \lnorm^-, \lnorm^+)$-expander graph with $\lnorm^- > -1$ and $\anorm \leq 1/5$.
We define
\begin{equation*}
	L \defined \log\p[\Big]{\frac{1 + \lnorm^+ + \anorm}{2 \anorm}},
\end{equation*}
and recall the main hypothesis that
\begin{equation*}
\tag{\ref{eq:main-condition}}
    \frac{32 \anorm (1 + 4\lnorm^+ - 3\lnorm^-) L}{(1 + \lnorm^-)^2} < 1,
\end{equation*}
First of all, we notice some numerical consequences of the hypotheses.
Since we assumed that $-1 < \lnorm^- \leq 0 \leq \lnorm^+$ and that $\anorm \leq 1/5$, we have $L \geq \log 3 > 1$ and $\anorm / (1 + \lnorm^-) < 1/32$.

As we have seen from~\eqref{eq:main-contradiction} in this section, it is enough to construct a sequence of angles $\pi/2 = \beta_0 > \beta_1 > \dotsb > \beta_M \geq 0$ and a sequence $c_0 < c_1 < \dotsb < c_M$ of bounds with $\card{\cC_{\beta_k}} \geq c_k$ for $0 \leq k \leq M$, such that $c_M \sin^2(\beta_M) > 5 \anorm^2 n / 2$.
For convenience, we set $\beta_{-1} = \pi$ and $\beta_{M+1} = 0$, $c_{-1} = 0$ and $c_{M+1} = n$.

We proceed by constructing the sequences $\beta_k$ and $c_k$.
Fix $\eps \defined (1 + \lnorm^-) / 2$ and $\delta \defined \eps/(\lnorm^+ - \lnorm^-) = (1 + \lnorm^-)/(2(\lnorm^+ - \lnorm^-))$.
We start with $\beta_0 \defined \pi/2$ and $c_0 \defined \card{\cC_{\pi/2}}$.

There are two stages in the construction of the angles, depending on how large we can guarantee $\card{\cC_{\beta_k}}$ to be.
Ideally, we would repeatedly use \Cref{lem:amplification-cancelation} from the start to construct the required sequence of angles.
However, the required lower bound on $\sin(\beta - \gamma)$ may be unfeasible when $\card{\cC_{\beta_k}}$ is not already of a reasonable size.
Therefore, we rely on \Cref{lem:amplification-simple} with $\rho = 1$ while $\card{\cC_{\beta_k}}$ is still small.
Indeed, that is always possible since $8 \anorm / (1 + \lnorm^-) < 1$ follows from~\eqref{eq:main-condition}, as discussed before.

Assume that we already have $\beta_k$ and $c_k$ defined, so we choose $\beta_{k+1}$ and give a lower bound $\card{\cC_{\beta_{k+1}}} \geq c_{k+1}$ as follows.

\begin{description}[leftmargin=0cm]
\item[Case 1 ($c_k < (1 + \lnorm^+ + \anorm) c_0 / 2 \anorm$)]
We set
\begin{align*}
\beta_{k+1} &\defined
    \beta_k - \sin^{-1} \p[\bigg]{ \frac{8\anorm}{1+ \lnorm^-} }, \\
c_{k+1} &\defined
    \min \set[\big]{ (1 + \delta) \card{\cC_{\beta_k}} ,\, \anorm n}.
\end{align*}
\item[Case 2 ($c_k \geq (1 + \lnorm^+ + \anorm) c_0 / 2 \anorm$)]
We set
\begin{align*}
\beta_{k+1} &\defined
    \beta_k - \sin^{-1} \p[\bigg]{ \frac{4(1 + \lnorm^+ + \anorm)c_0}{(1+ \lnorm^-) \card{\cC_{\beta_k}}} }, \\
c_{k+1} &\defined
    \min \set[\big]{ (1 + \delta) \card{\cC_{\beta_k}} ,\, n/2 }.
\end{align*}
\end{description}

By \Cref{lem:amplification-simple} with $\rho = 1$ and \Cref{lem:amplification-cancelation}, we indeed have that $\card{\cC_{\beta_{k+1}}} \geq c_{k+1}$ in both cases.
Recall that $\anorm \leq 1/5$, so $\anorm n \leq n/2$.

We define the critical step $k^\ast$ to be the smallest $k$ such that either $c_k \geq \anorm n$ or
\begin{equation*}
	c_k \geq (1 + \delta)^{k}c_0 \geq \frac{(1 + \lnorm^+ + \anorm)c_0}{2\anorm}.
\end{equation*}
This implies that $k^\ast \leq \ceil[\big]{ L / \log\p[\big]{1 + \delta} }$.
Now, either $L / \log(1 + \delta) \leq 1$ and $k^\ast = 1$, or
\begin{equation*}
    k^\ast
    \leq \max \set[\Big]{1, \frac{2L}{\log(1+\delta)} }
    \leq \max \set[\Big]{1, \frac{(2 + \delta) L}{\delta}} = \max \set[\Big]{1, \frac{(1 + 4 \lnorm^+ - 3\lnorm^-) L}{1 + \lnorm^-}},
\end{equation*}
where we used that $\ceil{x} \leq 2x$ for $x \geq 1$ and $\log(1 + x) \geq 2x/(2 + x)$ for $x \geq 0$.

The construction of the angles goes as follows.
We update $c_k$ with the Case 1 rules until we reach $k^\ast$.
If $c_{k^\ast} \geq \anorm n$, then we set $M = k^\ast$ and stop.
Otherwise, we update $c_k$ for $k > k^\ast$ using the Case 2 rule until $c_k = n/2$.
We then set $M$ so that $c_M = n/2$.
We claim that we always get $c_M \sin^{2}(\beta_M) > 5\anorm^2 n/2$.

\begin{description}[leftmargin=0cm]
\item[Case 1 (Stop with $M = k^\ast$ and $c_M = \anorm n$)]

Using the fact that $\sin^{-1}(x) \leq \pi x/2$ for any $0 \leq x \leq 1$, we give a lower bound on $\beta_M$ as follows,
\begin{align*}
	\frac{\pi}{2} - \beta_M &= \beta_0 - \beta_M = \sum_{k = 0}^{k^\ast-1} (\beta_{k} - \beta_{k+1}) \\
	&\leq k^\ast \sin^{-1}\p[\Big]{\frac{8 \anorm}{1 + \lnorm^-}}
	\leq \frac{\pi}{8} \p[\Big]{\frac{32 \anorm k^\ast}{1 + \lnorm^-}} < \frac{\pi}{8},
\end{align*}
where in the last step, we used the hypothesis~\eqref{eq:main-condition} and the upper bound above for $k^{\ast}$.
Indeed, if $k^\ast = 1$, recall that we had $\anorm / (1 + \lnorm^-)  < 1/32$.
If $k^\ast \leq (2 + \delta)L/\delta$, the inequality above is precisely condition~\eqref{eq:main-condition}.
Therefore, $\beta_M > 3\pi/8$ and $\card{\cC_{3\pi/8}} \geq \anorm n \geq 5 \anorm^2 n$, so we have $c_M \sin^2(\beta_M) > 5 \anorm^2 n / 2$.

\item[Case 2 (Stop with $M > k^\ast$ and $c_M = n/2$)]

Again, we need to give a lower bound on $\beta_M$.
Note that
\begin{align*}
	\frac{\pi}{2} - \beta_M = \beta_0 - \beta_M \leq \sum_{k = 0}^{k^\ast - 1}(\beta_k - \beta_{k+1}) + \sum_{k = k^\ast}^{M-1}(\beta_k - \beta_{k+1}).
\end{align*}
The first sum we bound as in Case 1, so we obtain
\begin{equation*}
    \sum_{k = 0}^{k^\ast - 1}(\beta_k - \beta_{k+1})
    \leq \frac{\pi}{8} \p[\Big]{\frac{32 \anorm k^\ast}{1 + \lnorm^-}} < \frac{\pi}{8}.
\end{equation*}
For the second sum, we observe that $\card{\cC_{\beta_k}} \geq c_k \geq (1 + \delta)^k c_0$, thus
\begin{align*}
    \sum_{k = k^\ast}^{M-1} (\beta_k - \beta_{k+1})
    &\leq \sum_{k = k^\ast}^{M-1} \sin^{-1} \p[\Big]{ \frac{4(1+\lnorm^+ + \anorm)c_0}{(1+ \lnorm^-) \card{\cC_{\beta_k}}} }
    \leq \sum_{k = k^\ast}^{M - 1} \frac{\pi}{2} \, \frac{4(1+\lnorm^+ + \anorm)c_0}{(1+ \lnorm^-) \card{\cC_{\beta_k}}} \\
    &\leq \frac{2 \pi (1+\lnorm^+ + \anorm)}{(1+\lnorm^-)} \sum_{k = k^\ast}^\infty (1 + \delta)^{-k}
    = \frac{2\pi (1+\lnorm^+ + \anorm)}{(1+\lnorm^-)} \, \frac{(1+\delta)}{\delta (1 + \delta)^{k^\ast}} \\
    &\leq \frac{4 \pi \anorm (1 + 2\lnorm^+ - \lnorm^-)}{(1 + \lnorm^-)^2}
    \leq \frac{\pi}{8}\p[\Big]{\frac{32 \anorm (1 + 4 \lnorm^+ - 3 \lnorm^-) L}{(1 + \lnorm^-)^2}}
    < \frac{\pi}{8}.
\end{align*}
We have used that $(1 + \delta)^{k^\ast} \geq (1 + \lnorm^+ + \anorm)/ 2\anorm$, and \eqref{eq:main-condition} on the last step.
This gives
\begin{equation*}
	\beta_M > \frac{\pi}{2} - \frac{\pi}{8} - \frac{\pi}{8} = \frac{\pi}{4}.
\end{equation*}
But $c_M \sin^2(\beta_M) > n \sin^2(\pi/4)/2 = n/4$, so $c_M \sin^2(\beta_M) > 5 \anorm^2 n / 2$ as $\anorm \leq 1/5$.
\end{description}

In any case, we have $c_M \sin^2(\beta_M) > 5 \anorm^2 n / 2$, so the proof is finished.
\end{proof}


\section{Application: Ramanujan graphs and random regular graphs}
\label{sec:numerics}

In this section, we indicate how to perform the amplification argument in the proof of \Cref{thm:amplification-main} numerically to obtain \Cref{thm:main-synchrony-regular}, namely that any $d$-regular $(n, d, \anorm)$-expander graph with $\anorm \leq 0.0816$ must be globally synchronizing.
As discussed in \Cref{cor:main-ramanujan} in the introduction, this implies that Ramanujan graphs, namely $d$-regular $(n,d,\anorm)$-expanders with $\anorm = 2 (\sqrt{d-1})/d$, are globally synchronizing if $d \geq 600$.
Furthermore, by a result of Friedman~\cite{Friedman2008-sf}, this also implies that the probability that a graph uniformly chosen between the $d$-regular graphs with $n$ vertices goes to one as $n$ goes to infinity, also when $d \geq 600$.

If we apply \Cref{thm:amplification-main}, then we would obtain the same conclusion for $\anorm \leq 0.0068$, which translates to $d$-regular Ramanujan graphs with $d \geq 85730$.
We perform the amplification argument numerically to push up our bound to $\anorm \leq 0.0816$ (or $d \geq 600)$.
We note that there is a hard barrier to our method at $\anorm \leq 0.24585$ (or $d \geq 66$), as we do not have a  result analogous to \Cref{lem:rhobounds}, even with the adaptations discussed below.
Obtaining global synchrony with $d \geq 3$ would require $\anorm \leq 0.943$, which is beyond our current techniques.

One source of the improvements comes from carefully tracking the consequences of using \eqref{eq:rho1-inequality-regular} instead of \eqref{eq:rho1-inequality} in \Cref{lem:rhobounds}.
This is what leads to the aforementioned limitation of $\anorm \leq 0.24585$ in place of $\anorm \leq 1/5$, see \Cref{rmk:regular}.
We also noticed that there is a gain in tweaking the parameters $\eps$ and $\rho$ in \Cref{lem:amplification-simple} and $\eps$ in \Cref{lem:amplification-cancelation} independently.
Another source of improvement is by using the actual values of $\sin^{-1}$ in the angle increments, instead of bounding via the inequality $\sin^{-1}(x) \leq \pi x / 2$.
To obtain a result independent of $n$, we finish off the amplification with the geometric sum as in Case 2 in the proof of \Cref{thm:amplification-main}.

For instance, to obtain synchronization for $\anorm \leq 0.0816$, we apply \Cref{lem:amplification-simple} three times with $\eps = 0.23$ and $\rho = 0.38$ and \Cref{lem:amplification-cancelation} three times with $\eps = 0.184$, after which we can perform the geometric series and finish the argument.
We obtain a sequence of angles guaranteed to satisfy $\beta_i \geq 0.132$ for all $i$.
This leads to the desired contradiction, and \Cref{thm:main-synchrony-regular} is proven.
The parameters obtained through this argument are exhibited in~\Cref{tab:amplification} and the details are provided below.

\begin{proof}[Proof of \Cref{thm:main-synchrony-regular}]
Let $G$ be a $d$-regular $(n,d,\anorm)$-expander with $\anorm = 0.0816$ and let $\theta$ be a stable state in $G$.
We may assume by a rotation that $\rho_1(\theta) \geq 0$.
We prove by contradiction that $\theta$ must be fully synchronized, which is equivalent by the half-circle lemma (\Cref{lem:half-circle}) to $\cC_{\pi/2} = \emptyset$.
Assume that $\card{\cC_{\pi/2}} \geq 1$.
Using \Cref{rmk:regular}, we obtain a refinement of the contradiction condition \eqref{eq:main-contradiction}.
Indeed, a contradiction is reached if there is an angle $0 < \xi < \pi/2$ such that $\card{\cC_\xi} \geq 5 \anorm^2 / 4 \sin^2(\xi)$.
In particular, we have a contradiction if
\begin{equation}
\label{eq:contradiction-concrete}
    \card{\cC_\xi} \geq \frac{0.00832 n}{\sin^2(\xi)}.
\end{equation}
We build a sequence of angles $\beta_i$ and a sequence of constants $c_i$ satisfying $\card{\cC_{\beta_i}} \geq c_i\card{\cC_{\pi/2}}$, starting with $\beta_0 = \pi/2$ and $c_0 = 1$. 
Note that \Cref{lem:amplification-simple} with $\eps = 0.23$ and $\rho = 0.38$ says that we have
\begin{equation}
\label{eq:amp1-concrete}
    \card{\cC_{\beta_{i+1}}} \geq \min\set[\big]{ 2.4093 \card{\cC_{\beta_i}} ,\; 0.031 n},
\end{equation}
as long as $\beta_{i} - \beta_{i+1} \geq \sin^{-1}(0.3272)$.
Whenever we apply this lemma, we obtain a new lower bound on $\card{\cC_{\beta_{i+1}}}$ given by \eqref{eq:amp1-concrete}.
We may assume that the lower bound is $2.4093 \card{\cC_{\beta_i}}$ as long as $0.031n \geq 0.00832n/ \sin^2(\beta_{i+1})$, since this would lead to a contradiction to \eqref{eq:contradiction-concrete}.
We then apply \Cref{lem:amplification-simple} for a total of three times to obtain that $\card{\cC_{\beta_3}} \geq c_3 \card{\cC_{\pi/2}}$ for $\beta_3 = \pi/2 - 3 \sin^{-1}(0.3272) \geq 0.570$ and $c_3 = 13.985$.

We now switch to \Cref{lem:amplification-cancelation} with $\eps = 0.184$, which states that as long as $\beta_{i} - \beta_{i+1} \geq \sin^{-1}\p[\big]{3.168 \card{\cC_{\pi/2}}/\card{\cC_{\beta_i}}}$, we have
\begin{equation}
\label{eq:amp2-concrete}
    \card{\cC_{\beta_{i+1}}} \geq \min\set[\big]{ 2.1274 \card{\cC_{\beta_i}} ,\; n/2}.
\end{equation}
We apply this result until we obtain a contradiction.
Assume that we always obtain the bound $\card{\cC_{\beta_{i+1}}} \geq 2.1274 \card{\cC_{\beta_i}}$ from \eqref{eq:amp2-concrete},
and we take $\beta_{i+1} = \beta_i - \sin^{-1}\p[\big]{3.168/c_i} \geq \beta_i -  3.326/c_i$, with $c_{i+1} = 2.1274 \, c_i$.
We have used that $\sin^{-1}(x) \leq 1.05 x$ for $x \leq 0.25$.
Therefore, this sequence will never produce an angle with $\beta_M \geq 0.131$ since
\begin{align*}
    \beta_M &\geq \beta_3 - \sum_{i=3}^{4}\sin^{-1}\p[\Big]{\frac{3.168}{c_i}} - \sum_{i=5}^{M-1} \frac{3.326}{c_i} \\
    &\geq 0.570 - \sum_{k=0}^{1}\sin^{-1}\p[\Big]{\frac{3.168}{13.985 \cdot 2.1274^k}} - \sum_{k=0}^{\infty} \frac{3.326}{63.295 \cdot 2.1274^k} \geq 0.131.
\end{align*}
If at any point in \eqref{eq:amp2-concrete}, we obtain $\card{\cC_{\beta_{i+1}}} \geq n/2$, we have a contradiction to \eqref{eq:contradiction-concrete}, since $\sin^2({\beta_{i+1}}) \geq \sin^2(0.131) \geq 0.008323 n / (n/2) = 0.016646$.
If that is not the case, $\card{\cC_{\beta_i}}$ grows without bound, which is also a contradiction.
\end{proof}

\begin{table}
\begin{tabular}{rccccccc}
\hline
$i$ & 0 & 1 & 2 & 3 & 4 & 5  & $\infty$ \\ \hline
$\beta_i$ & $\pi/2$ & 1.237 & 0.904 & 0.570 & 0.342 & 0.235 & 0.131 \\
$c_i$ & 1.0 & 2.409 & 5.804 & 13.985 & 29.752 & 63.295 & $\infty$ \\
Lemma & \ref{lem:amplification-simple} & \ref{lem:amplification-simple} & \ref{lem:amplification-simple} & \ref{lem:amplification-cancelation} & \ref{lem:amplification-cancelation} & \ref{lem:amplification-cancelation} & \\
$\eps$ & 0.23 & 0.23 & 0.23 & 0.184 & 0.184 & 0.184 & \\
$\rho$ & 0.38 & 0.38 & 0.38 & & & & \\
\hline\\
\end{tabular}
\caption{Parameters used and bounds obtained in the amplification argument for $\anorm \leq 0.0816$.}
\label{tab:amplification}
\end{table}

We stopped our efforts at $\anorm \leq 0.0816$ since this corresponds to $d \geq 600$.
While we could have pushed the numerical computations a bit further, we believe it unlikely that we could obtain anything past $\anorm \leq 0.1$ or $d \geq 400$.


\section{Application: Erd\H{o}s-Rényi random graphs}
\label{sec:erdos}

The goal of this section is to display the spectral properties of the \ER random graph $G(n,p)$.
The first result of this section is \Cref{cor:expansion-whp}, which gives us enough information to deduce the conjecture of Ling, Xu, and Bandeira (\Cref{thm:ling-xu-bandeira-conjecture}) from \Cref{thm:amplification-main}.
The second result is \Cref{thm:gnp-spectral}, which is a self-contained and quantitative version of \Cref{cor:expansion-whp}.

Recall that an \ER random graph is a graph on $n$ vertices where each of the $\binom{n}{2}$ edges is present independently with probability $p$.
One of the classical results of Erd\H{o}s and Rényi~\cite{Erdos1959-jh} concerns the connectedness of $G(n,p)$.
Indeed, they showed that if $p = (\log n + c_n)/n$ for some sequence $c_n$, then
\begin{align*}
    \prob[\big]{G(n,p) \text{ is connected}} \to \begin{cases}
    0 & \text{if $c_n \to - \infty$,} \\
    e^{-e^{-c}} & \text{if $c_n \to c$,} \\
    1 & \text{if $c_n \to +\infty$.}
    \end{cases}
\end{align*}

As mentioned in the introduction, connectedness is a necessary but not sufficient condition for synchronization.
Instead, it is essential to have a density $p$ above the connectivity threshold of $\log n / n$ to ensure global synchrony.
\Cref{thm:ling-xu-bandeira-conjecture} implies that this condition is also sufficient.
To prove this, we need to show that $G(n,p)$ has good spectral properties near the connectivity threshold and then apply \Cref{thm:amplification-main}.
Namely, we show that $G(n,p)$ is an $(n, \degm, \anorm, \lnorm^-, \lnorm^+)$-expander with $\degm = pn$ and suitable $\anorm$, $\lnorm^-$ and $\lnorm^+$.

Our first step is to control the operator norm of $\Delta_A$, the centered adjacency matrix of $G(n,p)$.
Random symmetric matrices with independent entries (apart from the symmetry) are a well-studied object in random matrix theory.
The entries of $\Delta_A$ are distributed as centered Bernoulli variables, which are non-symmetrically distributed.
If the distribution was symmetric, then all even moments would vanish and much of the machinery of \cites{Bandeira2016-gu,Seginer2000-he} would give us precise bounds.
Without this assumption, the current best result is due to Péché and Soshnikov~\cite{Peche2007-ig}, which implies that
\begin{equation}
\label{eq:norm-adjacency-peche}
    \norm{\Delta_A} \leq 2 \sqrt{np(1-p)} + o(n^{-1/44}),
\end{equation}
with high probability, as $n \to \infty$.
Thus we can take $\anorm = (2 + \eps) (pn)^{-1/2}$.

Since $\anorm$ is quite small, \Cref{prop:degree-spectral} implies that we only need a fine control of the minimum and maximum degree in $G(n,p)$ to obtain bounds for $\lnorm^-$ and $\lnorm^+$.
We take now $p = \gamma \log n / n$ for some $\gamma = \gamma(n) > 1$.

If $\gamma \to \infty $ as $n \to \infty$, then the degrees of $G(n,p)$ are well concentrated and we can take $\abs{\lnorm^\pm} = o(1)$.
Indeed, a standard application of Chernoff's bound, see for instance~\cite{Janson2000-al}*{Corollary 2.3}, gives
\begin{align}
    \prob[\big]{ \exists v \st  \abs[\big]{\deg(v) - pn} \geq \eps pn }
    &\leq n \prob[\big]{ \abs[\big]{\deg(v_1) - pn} \geq \eps pn } \nonumber \\
\label{eq:degree-chernoff}
    & \leq 2n \exp\p[\big]{ - \eps^2 pn / 3} = 2n^{1 - \eps^2 \gamma / 3}.
\end{align}
Therefore, for some $\gamma_0 > 1$ large enough, if $\gamma \geq \gamma_0$, then $G(n,p)$ is an $(n, \degm, \anorm , \lnorm^-, \lnorm^+)$-expander with $\degm = pn$, $\anorm = (2+\eps) (pn)^{-1/2}$ and $\lnorm^- = - \eps$, $\lnorm^+ = \eps$.
The conditions of \Cref{thm:amplification-main} then hold for $n$ large enough.
Therefore, $G(n,p)$ is globally synchronizing with high probability for $p \geq \gamma (\log n) /n$, with $\gamma \geq \gamma_0$.

We need to control the maximum and minimum degree when $\gamma$ is a constant near $1$ to close the gap to the connectivity threshold. 
It turns out that in this range, the order of the extremal degrees can be described as solutions to an implicit equation.
Indeed, from Bollobás~\cite{Bollobas2001-dl}*{Exercise 3.4}, we have that $G(n,p)$ satisfies, with high probability,
\begin{align*}
    \degmin &= (1 + o(1)) (1 + \lnorm^-) pn, \\
    \degmax &= (1 + o(1)) (1 + \lnorm^+) pn,
\end{align*}
where $-1 < \lnorm^- < 0 < \lnorm^+$ are the solutions in $c$ to the equation
\begin{equation}
\label{eq:implicit}
    \gamma \p[\big]{ (1 + c) \log(1 + c) - c} = 1.
\end{equation}

Let us do a quick analysis of the solutions.
Write $h(x) = (1+x) \log(1+x) - x$, so~\eqref{eq:implicit} is equivalent to $h(c) = 1/\gamma$.
One can observe that $h(x)$ is strictly decreasing from $1$ to $0$ on $(-1,0]$  and strictly increasing from $0$ to $\infty$ on $[0,\infty)$.
Therefore, there are indeed only two solutions to $h(c) = 1/\gamma$, given by smooth functions $c^- \from (1, \infty) \to (-1,0)$ and $c^+ \from (1, \infty) \to (0,\infty)$.
We can continuously extend $c^-$ to $(0,\infty)$ by setting $c^-(\gamma) = -1$ for all $0 < \gamma \leq 1$, see \Cref{fig:implicit-equation}.
One can see from implicit differentiation that $c^-$ is strictly increasing and $c^+$ is strictly decreasing.
Moreover, as $\gamma \to \infty$, we have $c^-(\gamma) \to 0$ and $c^+(\gamma) \to 0$, and at $\gamma =1$ we have $c^-(1) = -1$ and $c^+(1) = e - 1$.
Another useful fact is that $c^+(\gamma) > - c^-(\gamma)$.

\begin{figure}[!ht]
\centering
\begin{tikzpicture}[scale=1.0]
    \begin{axis}[
        axis x line* = middle,
        axis y line* = middle,
        axis line style={-stealth},
        ymin=-1.21,
        ymax=2.01,
        xmin=-0.1,
        xmax=12.21,
        xtick = {1, 2.5886},
        xticklabels = {1, $\qquad\quad\, \approx 2.5886$},
        every tick/.append style={black, very thick},
        x tick label style = {xshift={0.5em}},
        xlabel={$\gamma$},
        xlabel style={at={(ticklabel* cs:1)}, anchor=south west},
        ytick = {-1, 0 , 1, 1.71828},
        yticklabels = {-1, 0 , 1, $e-1$},
        grid=major,
        grid style={dashed, gray},
        width = 9.1cm,
    ]
\addplot[black, ultra thick]
table {%
0.3 3.54306932276416
0.4 2.9673147698562
0.5 2.59112147666862
0.6 2.32231904938731
0.7 2.11871000065199
0.8 1.95800873057143
0.9 1.82723791643318
1.0 1.71828182845905
1.1 1.62578070231894
1.2 1.54603905914561
1.3 1.47641836839866
1.4 1.41497951210928
1.5 1.36026230491334
1.6 1.31114414793283
1.7 1.26674637491751
1.8 1.22637041186059
1.9 1.18945317139717
2.0 1.1555352035005
2.1 1.12423751370051
2.2 1.09524439876043
2.3 1.06829054112996
2.4 1.04315117007571
2.5 1.01963446588079
2.6 0.997575628177673
2.7 0.976832195016485
2.8 0.957280313204524
2.9 0.938811740105448
3.0 0.921331413582705
3.1 0.904755467376621
3.2 0.889009598751144
3.3 0.874027716990204
3.4 0.859750817497825
3.5 0.846126038405867
3.6 0.833105865805072
3.7 0.820647460760078
3.8 0.80871208670098
3.9 0.797264620004372
4.0 0.786273129879512
4.1 0.775708516278008
4.2 0.76554419660829
4.3 0.755755833684109
4.4 0.746321098656605
4.5 0.737219463746517
4.6 0.728432020457895
4.7 0.71994131965994
4.8 0.711731230501196
4.9 0.703786815595643
5.0 0.696094220313038
5.1 0.688640574331782
5.2 0.681413903884212
5.3 0.674403053351282
5.4 0.667597615054273
5.5 0.660987866251732
5.6 0.654564712485575
5.7 0.648319636535305
5.8 0.642244652337254
5.9 0.636332263309162
6.0 0.630575424591898
6.1 0.624967508781385
6.2 0.619502274776531
6.3 0.614173839414423
6.4 0.608976651603354
6.5 0.603905468698363
6.6 0.59895533489352
6.7 0.59412156143103
6.8 0.58939970844967
6.9 0.584785568314781
7.0 0.580275150289282
7.1 0.575864666420276
7.2 0.571550518529182
7.3 0.567329286205037
7.4 0.563197715710961
7.5 0.559152709723155
7.6 0.555191317829336
7.7 0.55131072772167
7.8 0.547508257024776
7.9 0.543781345705611
8.0 0.540127549016959
8.1 0.536544530930847
8.2 0.533030058022246
8.3 0.529581993767093
8.4 0.526198293221938
8.5 0.522876998055424
8.6 0.519616231904507
8.7 0.516414196030659
8.8 0.513269165253451
8.9 0.510179484140889
9.0 0.507143563437557
9.1 0.504159876713295
9.2 0.50122695721649
9.3 0.498343394917427
9.4 0.495507833728297
9.5 0.492718968887548
9.6 0.489975544497226
9.7 0.487276351202889
9.8 0.484620224006422
9.9 0.482006040202894
10.0 0.479432717433225
10.1 0.476899211845078
10.2 0.474404516354957
10.3 0.471947659004985
10.4 0.469527701408362
10.5 0.467143737277897
10.6 0.464794891032425
10.7 0.462480316476299
10.8 0.460199195547481
10.9 0.457950737130053
11.0 0.455734175927292
11.1 0.45354877139168
11.2 0.4513938067085
11.3 0.449268587829876
11.4 0.447172442556324
11.5 0.445104719663091
11.6 0.443064788068721
11.7 0.441052036043461
11.8 0.439065870455281
11.9 0.437105716051405
} node[above,pos=0.6]{$c^+(\gamma)$};
\addplot [black, ultra thick]
table {%
0.001 -1
1.000 -1
1.001 -0.999902389083225
1.101 -0.981645202625203
1.201 -0.960432496790757
1.301 -0.939089102550111
1.401 -0.918343900212928
1.501 -0.898466650127306
1.601 -0.879547410466112
1.701 -0.86159510424848
1.801 -0.844580381506203
1.901 -0.828456284147243
2.001 -0.813168774123099
2.101 -0.798662231786223
2.201 -0.784882321001288
2.301 -0.771777427411344
2.401 -0.759299308261059
2.501 -0.74740330404825
2.601 -0.736048309203481
2.701 -0.725196614728656
2.801 -0.714813688036339
2.901 -0.704867927668996
3.001 -0.695330414400807
3.101 -0.686174670636232
3.201 -0.677376434312253
3.301 -0.66891345012349
3.401 -0.660765278899707
3.501 -0.652913124828864
3.601 -0.64533967959627
3.701 -0.638028982204268
3.801 -0.630966293117779
3.901 -0.624137981372888
4.001 -0.617531423340097
4.101 -0.61113491192114
4.201 -0.6049375750603
4.301 -0.598929302557148
4.401 -0.593100680271049
4.501 -0.587442930905321
4.601 -0.581947860648815
4.701 -0.57660781103418
4.801 -0.571415615445392
4.901 -0.566364559772399
5.001 -0.561448346768719
5.101 -0.556661063719086
5.201 -0.551997153069443
5.301 -0.547451385711406
5.401 -0.543018836648403
5.501 -0.538694862801491
5.601 -0.534475082740031
5.701 -0.530355358146242
5.801 -0.526331776843713
5.901 -0.522400637238451
6.001 -0.518558434037403
6.101 -0.514801845123803
6.201 -0.511127719481442
6.301 -0.507533066071226
6.401 -0.504015043573391
6.501 -0.500570950917562
6.601 -0.497198218530756
6.701 -0.493894400240347
6.801 -0.490657165775292
6.901 -0.487484293814414
7.001 -0.484373665535474
7.101 -0.481323258623364
7.201 -0.47833114169912
7.301 -0.475395469135955
7.401 -0.472514476230733
7.501 -0.469686474702702
7.601 -0.466909848493695
7.701 -0.464183049846315
7.801 -0.461504595638703
7.901 -0.45887306395636
8.001 -0.456287090883173
8.101 -0.453745367495322
8.201 -0.451246637043139
8.301 -0.448789692307208
8.401 -0.446373373116172
8.501 -0.443996564014699
8.601 -0.44165819207102
8.701 -0.439357224814306
8.801 -0.43709266829289
8.901 -0.434863565245097
9.001 -0.432668993375039
9.101 -0.430508063726357
9.201 -0.4283799191474
9.301 -0.426283732841849
9.401 -0.424218706999219
9.501 -0.422184071500097
9.601 -0.420179082691353
9.701 -0.418203022226901
9.801 -0.41625519596991
9.901 -0.414334932952652
10.001 -0.412441584390457
10.101 -0.410574522746467
10.201 -0.40873314084415
10.301 -0.406916851024697
10.401 -0.405125084346661
10.501 -0.403357289825354
10.601 -0.401612933709691
10.701 -0.399891498794325
10.801 -0.398192483765048
10.901 -0.396515402575586
11.001 -0.394859783854018
11.101 -0.393225170337166
11.201 -0.391611118331423
11.301 -0.390017197198562
11.401 -0.388442988865169
11.501 -0.386888087354438
11.601 -0.38535209833912
11.701 -0.383834638714523
11.801 -0.382335336190487
11.901 -0.380853828901363
} node[below,pos=0.56]{$c^-(\gamma)$};
\end{axis}
\end{tikzpicture}
\caption{The functions $c^-(\gamma)$ and $c^+(\gamma)$.}
\label{fig:implicit-equation}
\end{figure}

The asymmetry of $\lnorm^-$ and $\lnorm^+$ in this range of $G(n,p)$ is the reason why we control $\Delta_L$ as $\lnorm^- \degm I \preceq \Delta_L \preceq \lnorm^+ \degm I$ instead of an estimate of the type $\norm{\Delta_L} \leq \lnorm \degm$.
Indeed, for $p = \gamma \log n / n$ and $\gamma \in (1, 2.5886)$, one can only show that $\norm{\Delta_L} \leq \lnorm \degm$ for $\lnorm > 1$, see \Cref{fig:implicit-equation}.
This would not even imply connectivity, let alone expansion.

Nonetheless, piecing it all together, we obtain the following.

\begin{corollary}
\label{cor:expansion-whp}
For every $\varepsilon> 0$ and $p = p(n) \geq (1 + \eps) \log n / n$, the probability that $G(n,p)$ is an $(n, \degm, \anorm, \lnorm^-, \lnorm^+)$-expander with $\anorm \le (2+\varepsilon)(pn)^{-1/2}$ and $-1 < \lnorm^- < 0 < \lnorm^+ < e-1$ goes to one as $n$ goes to infinity.
\end{corollary}

Although this suffices to establish the conjecture of Ling, Xu, and Bandeira, the goal of the remainder of this section is to provide explicit bounds on the probability of global synchrony.

To start, we provide some simple and non-asymptotic replacements to the Péché and Soshnikov bound~\eqref{eq:norm-adjacency-peche}.
We will also obtain an explicit version of Bollobás's result.
The result follows when both these results are combined.

\begin{theorem}
\label{thm:gnp-spectral}
For any $0 < \eps < 1$ and $1 + \eps < \gamma < 1 + \eps^{-2}$, there is $C = C(\gamma, \eps)$ such that the following is true.
Let $p \defined \gamma \log n / n$, $\degm \defined pn$ and $\anorm \defined 8 (\gamma \log n)^{-1/2}$.
Define $-1 < \lnorm^- < 0 < \lnorm^+$ to be the unique roots of the equation
\begin{equation*}
\tag{\ref{eq:implicit}}
    \gamma \p[\big]{ (1 + \lnorm) \log(1 + \lnorm) - \lnorm} = 1.
\end{equation*}
Then $G(n,p)$ is an $(n, \degm, \anorm, \lnorm^-, \lnorm^+)$-expander graph with probability at least  \ $1 - C(\gamma,\eps)(\log n)^4 n^{-\eps}$.
In particular, $G(n,p)$ is globally synchronizing with the same probability.
\end{theorem}

\subsection{Concentration of the adjacency matrix}
\label{subsec:erdos-adjacency}

We present here a simple and standard symmetrization argument that recovers the correct behavior of $\norm{\Delta_A}$ up to a constant factor.

\begin{proposition}
\label{prop:wigner-bound}
For any $C > 2\sqrt{2}$, there is $n_0(C)$ such that the following holds for all $n \geq n_0(C)$ and every $p \in (0,1)$.
The adjacency matrix $A$ of $G(n,p)$ satisfies
\begin{equation*}
\Expec{\norm{\Delta_A}} \leq C \sqrt{ np(1-p)},
\end{equation*}
where $\Expec{\norm{\Delta_A}}$ denotes the expected value of the random variable $\norm{\Delta_A}$.
\end{proposition}
\begin{proof}
Consider an independent copy $A$ of $A'$.
From Jensen's inequality, we obtain
\begin{equation*}
\Expec{\norm{\Delta_A}} = \Expec{\norm{\Delta_A - \expec{\Delta_{A'}}}}
\leq \Expec{\norm{\Delta_A - \Delta_{A'}}} = \Expec{\norm{A - A'}}.
\end{equation*}
We define $W \defined A - A' \in \set{-1, 0, 1}^{V \times V}$.
Clearly $W$ is a symmetric matrix, with entries $W_{x,y}$ that are i.i.d. above the diagonal.
For each $x \neq y$, the random variable $W_{x,y}$ is the difference of two independent Bernoulli $p$ random variables, thus $\prob{W_{x,y} = -1} = \prob{W_{x,y} = 1} = p(1-p)$ and $\prob{W_{x,y} = 0} = p^2 + (1-p)^2$.
In particular, the distribution of $W_{x,y}$ is now symmetric, so the odd moments of $W_{x,y}$ vanish and the even moments can be easily computed.
Indeed, $\Expec{W_{x,y}^{2\ell}} = 2p(1-p)$ for all $\ell \geq 0$.

We now use the combinatorial moment method to bound the operator norm of $W$ \cite{Anderson2009-nh}, so we will bound the spectral norm via the trace using $\Expec{\norm{W}} \leq \p[\big]{\Expec{\trace W^{2\ell}}}^{1/2\ell}$.
We mainly follow the approach of Bandeira and van Handel~\cite{Bandeira2016-gu}.

Let $\cC_{2 \ell}$ be the set of \indef{even cycles} $u \in V^{2 \ell}$, that is, a cyclic sequence of vertices $u_1, \dotsc, u_{2\ell}$ with $u_i \neq u_{i+1}$ and such that every edge is traversed an even number of times .
Given $u \in \cC_{2 \ell}$ write $n_k(u)$ for the number of distinct edges visited precisely $k$ times by $u$.
\begin{align*}
\Expec{\trace W^{2\ell}} &= \sum_{u \in V^{2\ell}} \prod_{k \geq 1} \p[\big]{ \Expec{W_{x,y}^k} }^{n_k(u)} = \sum_{u \in \cC_{2 \ell}} \prod_{k \geq 1} \p[\big]{ 2p(1-p) }^{n_k(u)}\\
&= \sum_{u \in \cC_{2 \ell}} \p[\big]{ 2p(1-p) }^{ \sum_{k \geq 1} n_k(u)} \leq \p[\big]{ 2p(1-p) }^{\ell} \sum_{u \in \cC_{2 \ell}} \prod_{k \geq 1} 1 \\
&= \p[\big]{ 2p(1-p) }^{\ell} \card{\cC_{2 \ell}}.
\end{align*}

On the other hand, if $G$ is a symmetric matrix with i.i.d. standard Gaussian entries, then $\Expec{W_{x,y}^k} \geq 1$, thus
\begin{align*}
\Expec{\trace G^{2\ell}} &= \sum_{u \in V^{2\ell}} \prod_{k \geq 1} \p[\big]{ \Expec{G_{x,y}^k} }^{n_k(u)} \geq \card{\cC_{2 \ell}}.
\end{align*}
A standard application of Slepian's Lemma (see \cite{Bandeira2016-gu}*{Lemma 2.2}) delivers 
\begin{equation*}
\p[\big]{\Expec{\trace(G^{2\ell})} }^{1 / 2\ell} \leq n^{1/2\ell} \p[\big]{\Expec{\norm{G}^{2\ell}} }^{1 / 2\ell} \leq n^{1/2\ell} \p[\big]{ 2\sqrt{n} + 2\sqrt{2\ell} }.
\end{equation*}

Putting it all together, we have
\begin{equation*}
\Expec{\norm{\Delta_A}} \leq \sqrt{2p(1-p)}\p[\big]{2\sqrt{n} + 2\sqrt{2\ell}} e^{\log n / 2\ell}.
\end{equation*}
Now put $\ell = \alpha \log n$, so $\Expec{\norm{\Delta_A}} \leq f(\alpha,n) \sqrt{n p (1-p)}$, where
\begin{equation*}
f(\alpha, n) = 2 \sqrt{2} e^{1/2\alpha} \p[\Big]{1 + \sqrt{2 \alpha \log n / n}}.
\end{equation*}
The result follows as $f(\alpha, n)$ is decreasing for $n \geq 3$.
\end{proof}
\begin{remark}
\label{rem:wigner-bound}
Since the function $f(\alpha, n)$ in the above proof satisfies $f(2,4) \leq 7.91$, $f(3, 120) \leq 4.98$, $f(5,880) \leq 3.994$ and $f(25,450000) \leq 2.996$, we can take $n_0(8) = 4$, $n_0(5) = 120$, $n_0(4) = 880$ and $n_0(3) = 450000$.
\end{remark}

The lost factor of $\sqrt{2}$ is inevitable if we use symmetrization.
However, together with a concentration argument, this is more than enough for our purposes.
Indeed, we use the following result from the book of Boucheron, Lugosi, and Massart~\cite{Boucheron2013-li}*{Theorem 7.12}.

\begin{theorem}[Convex Lipschitz concentration]
\label{thm:lipschitz-concentration}
Let $X = (X_1, \dotsc, X_d) \in \RR^d$ be a vector with independent random variables taking values in the interval $[0,1]$ and let $f \from [0,1]^n \to \RR$ be a $1$-Lipschitz convex function.
Then, for all $t > 0$,
\begin{equation}
\label{eq:lipschitz-concentration}
\prob[\big]{ f(X) > \bfM f(X) + t } \leq 2 e^{-t^2/4}.
\end{equation}
\end{theorem}

Here $\bfM f(X)$ is the median of the random variable $f(X)$.
However, in \Cref{prop:wigner-bound} we control the expected value of the operator norm, not the median.
Once an exponential concentration like~\eqref{eq:lipschitz-concentration} holds, however, these quantities must be very close.
Indeed, $\abs{\bfM f(X) - \Expec f(X)} \leq 4$ follows from~\cite{Boucheron2013-li}*{Exercise 2.2}.

\begin{corollary}
\label{cor:gnp-adjacency}
For any $p \in (0,1)$ and $n \geq 1000$, the adjacency matrix $A$ of $G(n,p)$ satisfies
\begin{equation}
\label{eq:deviation-adjacency}
    \prob[\big]{ \norm{\Delta_A} \geq 4\sqrt{p(1-p) n} + 4+ t } \leq 2e^{-t^2/4}.
\end{equation}
Moreover, for every $\gamma > 0$, there is $n_0(\gamma)$ such that if $p = \gamma \log n / n$ and $n \geq n_0(\gamma)$, then we have
\begin{equation}
\label{eq:deviation-adjacency-gamma}
    \prob[\big]{ \norm{\Delta_A} \geq 8\sqrt{ \gamma \log n}} \leq 2n^{-\gamma}.
\end{equation}
\end{corollary}
\begin{proof}
We apply \Cref{thm:lipschitz-concentration} with $d = \binom{n}{2}$ and $f \from \RR^d \to \RR$ defined as $f(X_1, \dotsc, X_d) = \norm{M - pJ}$, where $M$ is the symmetric matrix whose upper diagonal is formed by $X_1, \dotsc, X_d$.
It is not hard to see that $f$ is convex and $1$-Lipschitz.
As discussed above, this implies that $\abs{\bfM f(X) - \Expec f(X)} \leq 4$.
Together with \Cref{prop:wigner-bound} (and \Cref{rem:wigner-bound}), we have that $\bfM \norm{\Delta_A} \leq 4 \sqrt{p(1-p)n}+4$ for $n \geq 1000$.
Thus,~\eqref{eq:deviation-adjacency} follows from \Cref{thm:lipschitz-concentration}.
When $p = \gamma \log n / n$, \eqref{eq:deviation-adjacency-gamma} follows if $n$ large enough as a function of $\gamma$.
\end{proof}

\subsection{Concentration of the Laplacian matrix}
\label{subsec:erdos-laplacian}

Denote by $X_k$ the number of vertices of degree precisely $k$ in $G(n,p)$.
We also denote by $X_{\leq k}$ and $X_{\geq k}$ the number of vertices with degree at most $k$ and at least $k$, respectively.
Even though $X_k$ is not a binomial random variable, as the degrees of different vertices are not independent, $X_k$ is well approximated by a $\Bin(n,f_k)$ variable, where $f_k \defined \prob[\big]{\Bin(n-1,p) = k}$.

The following two lemmas allow us to control $\expec{X_{\leq k}}$ and $\expec{X_{\geq k}}$ in terms of $\expec{X_k}$, as long as $k$ is far enough from the average degree.

\begin{lemma}
\label{lem:lower-expec}
For every $0 < c < 1$, $p \leq c/(1 - c^2)$ and $k \leq (1 - c) pn$, we have
\begin{align*}
    \expec{X_{\leq k} } \leq c^{-2} \expec{X_{k}}.
\end{align*}
\end{lemma}
\begin{proof}
For all $i \leq (1 - c)pn$, we have
\begin{align*}
    \frac{f_{i-1}}{f_i} &= \frac{\binom{n-1}{i-1} p^{i-1} (1-p)^{n-i} }{\binom{n-1}{i} p^i (1-p)^{n-i-1}} = \frac{i(1-p)}{(n-i)p} \\
    &\leq \frac{(1 - c)pn}{(n - (1-c)pn)p} = \frac{(1-c)}{1 - (1-c)p} \leq 1 - c^2.
\end{align*}
Here we used that the function $x \to x/(n-x)$ is increasing for $x \in (0,n)$ and that $p \leq c/(1-c^2)$ on the last step.
Hence
\begin{align*}
    \prob[\big]{\Bin(n-1,p) \leq k}
    &= \sum_{i=0}^{k} f_i
    \leq \sum_{i=0}^{k} (1 - c^2)^i f_{k}
    \leq c^{-2} f_{k} \\
    &\leq c^{-2} \prob[\big]{\Bin(n-1,p) = k}.
\end{align*}
Multiplying by $n$ both sides gives $\expec{X_{\leq k}} \leq c^{-2} \expec{X_{k}}$.
\end{proof}

\begin{lemma}
\label{lem:upper-expec}
For every $c > 0$ and $k \geq (1 + c) pn$, we have
\begin{align*}
    \expec{X_{\geq k} } \leq \p[\big]{1+c^{-1}} \expec{X_{k}}.
\end{align*}
\end{lemma}
\begin{proof}
For all $i \geq (1 +  c) pn$, we have
\begin{align*}
    \frac{f_i}{f_{i-1}} &=  \frac{(n-i)p}{i(1-p)} \leq \frac{(n - (1+c)pn)p}{(1 + c)pn (1-p)} = \frac{(1 - (1+c)p)}{(1 + c) (1-p)} \leq \frac{1}{1 + c}.
\end{align*}
As in the proof of \Cref{lem:lower-expec}, we have
\begin{align*}
    \prob[\big]{\Bin(n-1,p) \geq k}
    &= \sum_{i=k}^{n} f_i
    \leq \sum_{i=0}^{n - k} \p[\Big]{\frac{1}{1 + c}}^i f_{k}
    \leq \p[\Big]{\frac{1 + c}{c}} f_{k}.
\end{align*}
Multiplying by $n$ both sides gives $\expec{X_{\geq k}} \leq \p[\big]{1+c^{-1}} \expec{X_k}$.
\end{proof}

Finally, we control the probability that $G(n,p)$ has the required spectral properties for our amplification method to work.

\begin{proposition}
\label{thm:gnp-degree}
For any $0 < \eps < 1$ and $1 + \eps < \gamma < 1 + \eps^{-2}$, there exists a constant $C = C(\gamma, \eps) > 0$ such that the following is true.
Let $p \defined \gamma \log n / n$, $\degm \defined \gamma \log n$, $\anorm \defined 8 (\gamma \log n)^{-1/2}$.
Define $\lnorm^- = \lnorm^-(\gamma, \eps)$ as the unique root in $c < 0$ of the equation
\begin{equation}
\label{eq:implicit-eps-neg}
    \gamma \p[\big]{ (1 + c - \eps) \log(1 + c - \eps) - (c - \eps)} = 1 + \eps,
\end{equation}
and define $\lnorm^+ = \lnorm^+(\gamma, \eps)$ as the unique root in $c > 0$ of the equation
\begin{equation}
\label{eq:implicit-eps-pos}
    \gamma \p[\big]{ (1 + c + \eps) \log(1 + c + \eps) - (c + \eps)} = 1 + \eps.
\end{equation}
Then the probability that $G(n,p)$ is an $(n, \degm, \anorm, \lnorm^-, \lnorm^+)$-expander is at least
\begin{equation*}
     1 - C (\log n)^4 n^{-\eps}.
\end{equation*}
\end{proposition}
\begin{remark}
\label{rem:implicit-eps}
As $\eps > 0$ is taken to be arbitrarily small, the solutions $\lnorm^- = \lnorm^-(\gamma, \eps)$ and $\lnorm^+ = \lnorm^+(\gamma, \eps)$ to \eqref{eq:implicit-eps-neg} and \eqref{eq:implicit-eps-pos}, respectively, approach the solutions $\lnorm^-(\gamma)$ and $\lnorm^+(\gamma)$ of equation \eqref{eq:implicit}, described in \Cref{fig:implicit-equation}.
Moreover, as we assumed that $1 + \eps < \gamma < 1 + \eps^{-2}$, it holds that $-1 < \lnorm^-(\gamma) \leq \lnorm^-(\gamma, \eps) < 0 < \lnorm^+(\gamma, \eps) \leq \lnorm^+(\gamma)$.
Indeed, we used that $1 + \eps^{-2} \leq (1+\eps)/\p[\big]{(1-\eps) \log(1-\eps) + \eps}$ for $0 < \eps < 1$.
\end{remark}
\begin{proof}
By \Cref{cor:gnp-adjacency}, we already have that $G(n,p)$ has $\norm{\Delta_A} \leq \anorm$ with probability $\geq 1 - 2n^{-\gamma}$.
In view of \Cref{prop:degree-spectral}, we have that $G(n,p)$ is an $(n, \degm, \anorm, \lnorm^-, \lnorm^+)$-expander if we prove that
\begin{equation}
\label{eq:gnp-degree-goal}
    (1 + \lnorm^- + \anorm) pn
    \leq \degmin(G) \leq \degmax(G)
    \leq (1 + \lnorm^+ - \anorm) pn.
\end{equation}

\noindent Let $k = \delta \log n$ for some $\delta > 0$ to be chosen later and observe that
\begin{align*}
    \expec{X_k} &= n \binom{n-1}{k} p^k (1-p)^{n-k-1} \leq n \p[\Big]{\frac{n^k}{k!}} p^k e^{-pn + p(k+1)} \\
    &\leq \frac{n^{1 - \gamma} (\gamma \log n)^k e^{2pk}}{k!},
\end{align*}
where we used that $(1-x) \leq e^{-x}$ and that $\binom{n}{k} \leq n^k / k!$.
Let $L_k \defined (\log n)^k / k!$, so
\begin{align*}
    \log L_k
    &= k \loglog n - \log k! \\
    &\leq \delta \log n \loglog n - \delta \log n \log (\delta \log n) + \delta \log n \\
    &= \p[\big]{ \delta - \delta \log \delta }\log n,
\end{align*}
where we used that $k! \geq (k/e)^k$.
Finally, setting $\delta \defined (1 + c \pm \eps)\gamma$, we have
\begin{align*}
    \expec{X_k} &\leq \exp\p[\Big]{\p[\big]{1 - \gamma + \delta - \delta \log \delta + \delta \log\gamma}\log n + 2pk} \\
    &\leq \exp\p[\Big]{\p[\big]{1 - \gamma\p[\big]{ (1 + c \pm \eps)\log(1 + c \pm \eps) - (c \pm \eps) } }\log n  + 2pk}.
\end{align*}
Now define $k^- \defined \floor{(1 + \lnorm^- - \eps) \gamma \log n}$ and $k^+ \defined \ceil{(1 + \lnorm^+ + \eps) \gamma \log n}$, where $-1 < \lnorm^- < 0 < \lnorm^+$ are the solutions to~\eqref{eq:implicit-eps-neg} and~\eqref{eq:implicit-eps-pos} as above.
This gives
\begin{align*}
    \expec{X_{k^\pm}} &\leq \exp\p[\big]{ -\eps \log n + 2pk^\pm + 4\loglog n},
\end{align*}
where the extra $4 \loglog n$ appeared as formally we took $k = \delta \log n + \eta$, with $\abs{\eta} \leq 1$ giving the rounding up or down to $k^+$ and $k^-$ respectively.
To see this, observe that the upper bound $\expec{X_k} \leq \exp\p[\Big]{ (1 - \gamma) \log n + k \log\gamma + \log L_k + 2pk}$ varies by at most $\log \gamma + \loglog n + \log(k+1) + 2p \leq 4\log\log n$, as $\abs{L_k - L_{k + \eta}} \leq \abs{\eta} \log\log n + \log(k+1)$.

Now we apply \Cref{lem:lower-expec} to control the expected value of $X_{\leq k^-}$ by the expected value of $X_{k}$.
We obtain by Markov's inequality, that
\begin{align*}
    \prob[\big]{\degmin \leq k^-}
    &= \prob[\big]{X_{\leq k^-} \geq 1}
    \leq \expec[\big]{X_{\leq k^-}}
    \leq \frac{1}{(\lnorm^- + \eps)^2} \expec[\big]{X_{k^-}} \\
    &\leq \frac{1}{(\lnorm^- + \eps)^2} \exp\p[\Big]{ - \eps\log n  + 2pk^- + 4\loglog n}.
\end{align*}
We do the same for the maximum degree,
\begin{align*}
    \prob[\big]{\degmax \geq k^+}
    &= \prob[\big]{X_{\geq k^+} \geq 1}
    \leq \expec[\big]{X_{\geq k^+}}
    \leq \frac{1 + \lnorm^+ - \eps}{\lnorm^+ - \eps} \expec[\big]{X_{k^+}} \\
    &\leq \frac{1 + \lnorm^+ - \eps}{\lnorm^+ - \eps} \exp\p[\Big]{ - \eps\log n + 2pk^+ + 4\loglog n}.
\end{align*}
Therefore, we obtain that~\eqref{eq:gnp-degree-goal} does not hold with probability at most
\begin{align*}
    \p[\Big]{\frac{1}{(\lnorm^- + \eps)^2} + \frac{1 + \lnorm^+ - \eps}{\lnorm^+ - \eps} } (\log n)^4 n^{-\eps} e^{2pk^+} + 2n^{-\gamma}
    \leq C(\gamma, \eps) (\log n)^4 n^{-\eps},
\end{align*}
as we wanted.
Notice that we use $\anorm \leq \eps$ for large enough $n$, which we can assume by increasing $C(\gamma, \eps)$.
\end{proof}

\begin{proof}[Proof of \Cref{thm:gnp-spectral}]
From \Cref{thm:gnp-degree} and \Cref{rem:implicit-eps}, we know that $G(n,p)$ is an $(n, \gamma \log n, 8(\gamma \log n)^{-1/2}, \lnorm^-(\gamma), \lnorm^+(\gamma))$-expander with probability $\geq 1 - C(\log n)^4 n^{-\eps}$.
Such a graph is globally synchronizing by \Cref{thm:amplification-main}.
\end{proof}


\section{Discussion and open questions}
\label{sec:generalizations}

One of the consequences of our main result is to settle the conjecture of Ling, Xu, and Bandeira~\cite{Ling2019-yn}.
In a nutshell, the conjecture says that \ER random graphs globally synchronize right above the connectivity threshold.
A natural challenge is to close the gap between connectivity and synchronization in the critical window, i.e., when $pn = \log n + c_n$ with $c_n \gg 1$.
In our result, $c_n = \eps \log n$ for a fixed $\eps > 0$.
A careful look at the computations allows us to take $\eps$ going to zero, but we do not establish any convergence rate so it would not be sharp.

In fact, we conjecture that a stronger hitting time version of Ling, Xu, and Bandeira's result holds.
Consider the random graph process $\set{G(n,m)}_{m = 0}^{n(n-1)/2}$, which is a sequence of random graphs on $n$ vertices with $G(n,0)$ being empty and $G(n,m+1)$ obtained by adding to $G(n,m)$ an edge that is chosen uniformly from the set of missing edges in $G(n,m)$.
Let $\tau = \min\set{m \st G(n,m) \text{ is connected}}$, that is, the moment that the random graph process becomes connected.
We believe that $G(n,m)$ is globally synchronizing for all $m \geq \tau$, completely sealing the gap between connectivity and synchrony.
A followup article by Jain, Mizgerd, and Sawhney~\cite{Jain2025-lr} has established this conjecture.

Despite our result for $G(n,p)$ being asymptotically sharp, we did not optimize it for any practically relevant range of $n$, particularly near the connectivity threshold.
For instance, if we take $p = 2 (\log n) / n$, we may need an impractically large $n$.

Another important consequence of our techniques is that random $d$-regular graphs synchronize for sufficiently large but constant degree $d$.
Following the same intuition, one expects that a random $d$-regular graph synchronizes when $d \geq 3$ because it matches the connectivity threshold for such a graph.
However, we extensively used \Cref{lem:rho1-bound,lem:rho1-sin} to start our amplification argument; these Lemmas do not give any useful bounds unless $d \geq 66$, as discussed in \Cref{sec:numerics}, which places a limit on what can be theoretically achieved without new ideas.

We suspect that our results may admit other generalizations.
Examples may include taking the oscillators in higher-dimensional spheres and considering other classes of graphs.
For instance, one may look at random graphs with non-uniform edge distribution or with edges distributed according to some underlying geometric structure, as considered in \cites{Diaz-Guilera2009-fe,Leyva2011-gd,Estrada2015-ha,Abdalla2022-ai}.

An interesting direction is to investigate whether these techniques can be generalized to other optimization landscapes arising in statistical inference problems, for example, the Burer--Monteiro relaxation in community detection~\cite{Bandeira2016-qr}.
The main issue in this particular case is that it would require us to consider signed adjacency matrices.
We used many times (indirectly) that the entries of the adjacency matrix are non-negative.
We note that since the first version of this paper appeared, progress has been made on  this problem.
A higher-rank version was settled by the first and second author, Boumal and McRae in \cite{mcrae2025nonconvex}, using completely different tools from optimization.
In further work, Rakoto Endor and Waldspurger~\cite{Endor2024-uo} proposed a clever dual-certificate technique to reduce the rank with near-optimal estimates.
McRae combined the techniques from \cites{mcrae2025nonconvex,Endor2024-uo} to establish optimal estimates that gives an alternative proof of \Cref{thm:ling-xu-bandeira-conjecture}.
On the other hand, it is not clear if this new line of work can recover a general result as our \Cref{thm:amplification-main}.

Another interesting question we have not considered is how quickly does a random initial state synchronize.
We believe that our techniques can be used to show that synchronization occurs rather rapidly for sufficiently expanding graphs.

To conclude, we mention an important research line proposed by Wiley, Strogatz, and Girvan~\cite{Wiley2006-fz} about the size of the basin of attraction for the synchronous state.
There may be situations where families of graphs are not globally synchronizing, but have a form of \indef{asymptotically global synchrony}, in the sense that the aggregate size of the basins of attraction of the spurious local minima is small, vanishing as the number of vertices of the underlying graphs grows.

\section*{Acknowledgments} We thank Nicolas Boumal and Philippe Rigollet for helpful discussions.
We thank the anonymous referees for the valuable comments that improved our manuscript. 

Most of the work was done while the first author was affiliated to ETH Z\"urich. The first author is supported by NSF and Simons Research Collaborations on
the Mathematical and Scientific Foundations of Deep Learning.
The third author was supported by NSF (DMS-2319371).
The fourth author was partially supported by ERC Starting Grant 101163189 and UKRI Future Leaders Fellowship MR/X023583/1.
The last author is grateful for support from DARPA (DARPA-PA-24-04-07), DOE (DE-AC52-07NA27344), Office of Naval Research (N00014-23-1-2729), NSF CAREER (DMS-2045646), and NSF (DMS-2319621).




\end{document}